\documentclass[reqno,12pt]{amsart}
\usepackage{amssymb,amsmath,amsthm,}
    \usepackage{a4wide}
\usepackage{amscd}
\usepackage{amsfonts}
\usepackage{amssymb}
\usepackage{latexsym}
\usepackage{color}
\usepackage{esint}
\usepackage{graphicx}
\usepackage{float}
\graphicspath{{Figures/}}

\usepackage{graphicx}
\usepackage[breaklinks=true,bookmarks=false]{hyperref}
\usepackage{cleveref}

\usepackage[makeroom]{cancel}

\newtheorem{theorem}{Theorem}

\newtheorem{definition}{Definition}
\newtheorem{lemma}{Lemma}

\newtheorem{remark}{Remark}

\let\e=\varepsilon

\let\h=v

\let\p=\partial

\let\O=\Omega

\let\o=\omega

\numberwithin{equation}{section}

\let\hide\iffalse
\let\unhide\fi

\newcommand{\nablaA}{\nabla^{\text{a}}}
\newcommand{\nablaS}{\nabla^{\text{sym}}}

\newcommand{\R}{\mathbb{R}}
\renewcommand{\P}{\mathbf{P}}

\newcommand{\be}{\begin{equation}}
\newcommand{\bm}{\begin{multline}}
\newcommand{\ee}{\end{equation}}
\newcommand{\dd}{\mathrm{d}}

\newcommand{\Div}{\text{div} \, }

\newcommand{\Bes}{\begin{eqnarray*}}
\newcommand{\Ees}{\end{eqnarray*}}
\newcommand{\Be}{\begin{equation} }
\newcommand{\Ee}{\end{equation}}

\newtheorem{coro}{Corollary}

\def\p{\partial}

\def\O{\Omega}
\def\R{\mathbb{R}}

\def\B{\begin{equation}}
\def\E{\end{equation}}
\def\BN{\begin{eqnarray*}}
\def\EN{\end{eqnarray*}}

\begin{document}
\title{Macroscopic estimate of the linear Boltzmann and Landau equations with Specular reflection boundary}

 \author{Hongxu Chen}
\address{Department of Mathematics, The Chinese University of Hong Kong, Shatin, N.T., Hong Kong, email:  hongxuchen.math@gmail.com}
 \author{Chanwoo Kim}
 \address{Department of Mathematics, University of Wisconsin-Madison, Madison, WI, 53706, USA, email: chanwookim.math@gmail.com}
 
\date{\today}

\begin{abstract} In this short note, we prove an $L^6$-control of the macroscopic part of the linear Boltzmann and Landau equations. This result is an extension of the test function method of Esposito-Guo-Kim-Marra~\cite{EGKM}\cite{EGKM2} to the specular reflection boundary condition, in which we crucially used the Korn's inequality \cite{DV2} and the system of symmetric Poisson equations \cite{Bernou}. 

\end{abstract}

\maketitle

\section{Introduction}
In this paper, we consider the scaled collisional kinetic models 
\begin{equation*}
\e^s\p_t F + v\cdot \nabla_x F = \frac{1}{\e^k} Q_\alpha(F,F) \ \ \text{for} \ \  (t,x,v) \in \R_+ \times \O \times \R^3,    
\end{equation*}
with general parameters $\e^s$ and $\e^k$ with $s, k\geq 0$ which correspond to \textit{Strouhal number} and \textit{Knudsen number}, respectively. In our study, the spatial domains $\O$ 
is an open, {\color{black} $C^1$} and bounded subset $\O \subset \R^3$. {\color{black}To describe the boundary condition, we denote the boundary of the phase space as
\begin{equation}\label{gamma_bdr}
\gamma:= \{(x,v)\in \p\O \times \mathbb{R}^3\}.
\end{equation}
Let $n=n(x)$ be the outward normal direction at $x\in\p\O$. We decompose $\gamma$ as
\begin{equation*}
\begin{split}
    &\gamma_- = \{(x,v)\in \p\O\times \mathbb{R}^3 : n(x)\cdot v < 0\},  \\
    & \gamma_+ = \{(x,v)\in \p\O\times \mathbb{R}^3 : n(x)\cdot v > 0\}, \\
    &\gamma_0 = \{(x,v)\in \p\O\times \mathbb{R}^3 : n(x)\cdot v = 0\}.
\end{split}    
\end{equation*}}
We impose the specular reflection boundary condition for the incoming part at the boundary $\p\O$:
\begin{equation*}
F (t,x,v)  = F(t,x,R_x v ) \ \ \text{on} \ \  {\color{black}(x,v) \in \gamma_-},  
\end{equation*}
where the reflection operator is given as $R_xv = v- 2 n(x) (n(x) \cdot v)$. Throughout this paper we assume that 
\Be\label{n_cover_all}
\{n(x) : x \in \p\O\} = \mathbb{S}^2.
\Ee
We remark that such a condition excludes periodic-in-$x_3$ cylindrical domains in our consideration (cf. see \cite{KKL, kim2018decay}). 

Two most important models are the Boltzmann equations ($\alpha=B$) and Landau equations $(\alpha = L)$ where corresponding collision operators are respectively given by 
\Be\label{Q_B}
Q_B (F,G):=\int_{\mathbb{R}^3}\int_{\mathbb{S}^2}
|(v-u) \cdot \o|
     \big[F(u')G(v')-F(u)G(v) \big]\dd \omega \dd u,
\Ee
\Be\label{Q_L}
Q_L (F,G):= \nabla_v \cdot \Big\{ \int_{\mathbb{R}^3} \phi(v-u) [F(u)\nabla_v G(v) - G(v) \nabla_v F(u)] \dd u \Big\}.
\Ee
Here, \hide the Boltzmann collision kernel takes a form of 
\Be\notag
B(v-u, \o) = |v-u|^\gamma q_0 \left( \o \cdot \frac{v-u}{|v-u|} \right),
\Ee
with $-3 < \gamma \leq 1$ and a non-negative function $q_0$; while \unhide
the Landau collision kernel for the Coulombic interaction is defined as 
\begin{equation}\label{coulomb}
\phi(z):= \Big\{\mathbf{I} - \frac{z}{|z|}\otimes \frac{z}{|z|} \Big\}\cdot |z|^{-1}.    
\end{equation}


\subsection{Linear problem}
In this paper, we study perturbation solutions around a generic equilibrium, so-called the global Maxwellian: 
\Be\label{Maxwellian}
\mu (v) =  \frac{1}{ (2 \pi)^{3/2} } e^{- \frac{|v|^2}{2}}.
\Ee
More precisely, we are interested in a priori estimate of a family of solutions whose perturbation has a size of the Mach number $\e^m$ with $m\geq 0$:
\begin{align*}
F = \mu + \e^m f  \sqrt{\mu}.    
\end{align*}
Then $f$ should solve the following singular perturbation equations with the specular reflection boundary condition: 
\Be
\p_t f + \frac{1}{\e^s} v\cdot \nabla_x f + \frac{1}{\e^{k+s}} \mathcal L_\alpha f =  \frac{1}{\e^{k+s - m}} \frac{1}{\sqrt \mu} Q_\alpha ( f \sqrt \mu, f \sqrt \mu). \label{eqn_f}
\Ee
Here, the linear operator for Boltzmann ($\alpha= B$ and \eqref{Q_B}) or Landau ($\alpha=L$ and \eqref{Q_L}) equals 
\Be\label{L_alpha}
 \mathcal L_\alpha f = \frac{-1}{\sqrt \mu} \left[Q_\alpha(\mu, \sqrt{\mu} f)
+ Q_\alpha( \sqrt{\mu} f, \mu )
\right] \ \ \text{for} \ \  \alpha  \in  \{B , L\}.
\Ee
This linear operator $\mathcal{L}_\alpha$ has $5$-dimensional null space $\mathcal{N}$, which is spanned by 
\begin{equation}\label{basis_chi}
\begin{split}
 \chi_0(v)  : = \sqrt{\mu} ; \ \ 
 \chi_i(v)     := v_i \sqrt{\mu} \ \text{ for }i=1,2,3 ; \ \
 \chi_4  :=  \frac{|v|^2 - 3}{\sqrt{6}}\sqrt{\mu}.
\end{split}
\end{equation}
We define the $L^2_v$-projection on this null space $\mathcal{N}$ as 
\begin{equation}\label{projection}
\mathbf{P}f : = \sum_{i=0}^4 \langle f , \chi_i \rangle \chi_i = a(x)\chi_0 + \sum_{i=1}^3 b_i(x)\chi_i + c(x)\chi_4,
\end{equation}
where
\begin{equation}
\begin{split}
    \langle f,g\rangle     = \int_{\mathbb{R}^3} f(v)g(v)\dd v ,\label{velocity_inner}\\
   a(x)   : = \langle f,\chi_0\rangle ,  \ \
   b_i(x)   := \langle f,\chi_i\rangle
   \ \text{ for }
   i=1,2,3 ; \ \
   c(x)   := \langle f,\chi_4 \rangle.  
\end{split}
\end{equation}
 The linear operator is non-negative: 
\[
 \langle  \mathcal L_\alpha  f, f \rangle \geq 0.
 \]
 \hide for $\alpha \in \{B, L \}$, there exists $\sigma_\alpha >0$ such that 
 \Be
\langle  \mathcal L_\alpha  f, f \rangle  \geq \sigma_\alpha \|   (\mathbf{I} - \mathbf{P})f \|^2_{\alpha} .
\Ee
where we will specify a dissipation norm $\| \cdot \|_{\alpha}$ for $\alpha \in \{B, L \}$ later in \eqref{} and \eqref{}, respectively.\unhide
 
We note that the collision invariant implies 
\Be\label{collision_inv}
\P \mathcal{L}_\alpha =0 = \P  \left(\frac{1}{\sqrt \mu} Q_\alpha (\cdot \sqrt \mu, \cdot \sqrt \mu)\right).
\Ee

In this paper, we study the linear equations of \eqref{eqn_f}:
\Be\label{linear_bolzmann}
\p_t f + \frac{1}{\e^s} v\cdot \nabla_x f + \frac{1}{\e^{k+s}}  \mathcal L _\alpha f =  g, 
\Ee
\begin{equation}\label{specular_f}
f(t,x,v) = f(t,x,R_x v) \text{ on } x\in \p\O,
\end{equation}
where $g= g(t,x,v)$ is a given function satisfying 
\Be\label{Pg=0}
\mathbf{P} g =0.
\Ee
Here, $\P$ has been defined in~\eqref{projection}.

\hide

{\color{black} [Where did you use this? Please move this to the place used, otherwise delete it:]
}


For the linear Boltzmann operator 
\Be
\|  h \|_B:=  \int_{  \R^n}  \nu (v) |h(v) |^2   \dd v  .
\Ee
For the linear Landau operator 
\Be
\|  \cdot  \|_L:= {\color{black}\text{write it}}
\Ee
\unhide

\subsection{Angular momentum}
We define a finite-dimensional set of rigid motions on $\O$ (\cite{DV2}):
\begin{align}\label{infinitesimal_rigid}
    \mathcal{R}(\O) : =  \{x\in \O \mapsto M x+ x_0\in \mathbb{R}^3: M\in A_3 (\R) , x_0\in \mathbb{R}^3\},
\end{align}
where $A_3(\mathbb{R}): = \{ M = (m_{ij}): m_{ij} \in \R \  \text{for} \  i,j=1,2,3 \ \text{and} \  M + M^T=0\}$ is the 3 dimensional vector space of antisymmetric $3 \times 3$ real matrices. We also define the infinitesimal rigid displacement fields preserving $\O$:
\begin{align}\label{RO}
    \mathcal{R}_\O = \{R\in \mathcal{R}(\O): x_0 = 0, R(x)\cdot n(x) = 0 \text{ for any } x\in\p\O\}.
\end{align}
Recall that $n(x)$ is the outward normal at $x \in \p\O$. Clearly, as long as $\p\O \neq \emptyset$, we have that $\text{dim} (\mathcal{R}_\Omega)$ equals $0,1,$ or $2$.

\begin{definition}\label{def:basis}
For an open bounded subset $\O \subset \R^3$ with $\p\O \neq \emptyset$, we consider a dimension of $\mathcal R_\Omega$ defined in \eqref{RO}. We define 
\begin{equation}\label{def:axiss}
\begin{cases} \text{$\O$ is \textbf{non-axisymmetric} if }\text{dim} (R_\Omega) =0,\\
\text{$\O$ is \textbf{axisymmetric} if }\text{dim} (R_\Omega)  \neq 0 .
\end{cases}
\end{equation}
For an axisymmetric domain, we take $R_1(x)$ and $R_2(x)$ to be an orthonormal basis of $\mathcal{R}_\O$, so that for any $R(x)\in \mathcal{R}_{\O}$, we can express 
\begin{align}
    & R(x) = C_1 R_1(x) + C_2R_2(x) \text{ for some } C_1,C_2, \ R_1(x)\cdot R_2(x) = 0.     \label{basis}
\end{align}
Our convention follows: If $\dim(\mathcal{R}_\O)=1$, then $R_1(x) = R_2(x) \neq 0$; If $\dim(\mathcal{R}_\O)=0$, then $R_1(x)=0 = R_2(x)$.
\end{definition}

With a specular boundary condition, the angular momentum preserves when the domain is axisymmetric: using \eqref{collision_inv}, \eqref{velocity_inner} and \eqref{linear_bolzmann}, we derive that 
\begin{align}
&\frac{d}{dt} \int_{\O } R(x) \cdot b(t,x )  \dd v \dd x =\frac{d}{dt} \iint_{\O\times \mathbb{R}^3} (R(x) \cdot v) f \chi_{0}(v)  \dd v \dd x    \notag\\
&= - \frac{1}{\e^s} \iint_{\O\times \mathbb{R}^3} (R(x) \cdot v) (v\cdot \nabla_x f) \chi_{0}(v)  \dd v \dd x \notag\\
& = \frac{1}{\e^s}  \iint_{\O\times \mathbb{R}^3}  v\cdot \nabla_x (R(x) \cdot v) f \chi_{0}(v)  \dd v \dd x - \frac{1}{\e^s}  \iint_{\partial \O\times \mathbb{R}^3} (R(x) \cdot v)   (n(x) \cdot v) f \chi_{0}(v)  \dd v \dd S_x \notag.
\end{align}
Since $R(x) = Mx$ for some antisymmetric matrix $M$, we have $\nabla_x (R(x) \cdot v)=0$, thus the bulk term in the last line equals zero. Now we consider the second term. Using $I= n(x) \otimes n(x) + (I -n(x) \otimes n(x) )$ and the fact that $f(v)|_{\partial\Omega}$ is even in $n(x) \cdot v$ from the specular reflection boundary condition \eqref{specular_f}, we derive that
\begin{align*}
&\int_{  \mathbb{R}^3} (R(x) \cdot v)  ( n(x) \cdot v) f \chi_{0}  \dd v \bigg|_{\p\O}   \notag\\
& = (R(x) \cdot n(x)) \left(\int_{  \mathbb{R}^3}    (n  \cdot  v)   (n \cdot v) f \chi_{0}   \dd v\right)+  R(x) \cdot  \left(\int_{  \mathbb{R}^3}  \big(v - (n \cdot v) n \big)  ( n  \cdot v) f \chi_{0}  \dd v\right).
\end{align*}
The first term vanishes from $R(x) \cdot n(x)|_{\partial\Omega}=0$ in \eqref{RO}. The second term also vanishes as $ \big(v - (n(x) \cdot v) n(x) \big)   n (x) \cdot v f(x,v) \chi_{0}(v)$ is odd in $(n(x) \cdot v)$ from the specular reflection boundary condition of $f(x,v)$ (i.e. $f(x,v)|_{\partial\Omega}$ is even in $(n(x) \cdot v)$).
Thus we conclude that the total angular momentum is invariant in time:
\begin{align}
  \int_{\O}    R(x)\cdot \p_t b(t,x) \dd x = 0.\label{angular_momentum_preserve}
\end{align}

In this paper, without loss of generality, we assume the initial datum satisfies
\begin{align}
    & \int_{\O} R(x) \cdot \int_{\mathbb{R}^3} v f(0,x,v) \chi_0(v) \dd v \dd x = 0 \text{ for all } R(x)\in \mathcal{R}_\O. \label{angular_momentum_initial}
\end{align}
Then the conservation law~\eqref{angular_momentum_preserve} leads to
\begin{align}
  &\int_{\O} R(x)\cdot \int_{\mathbb{R}^3} v f(t) \chi_{0}(v)  \dd v \dd x  =  \int_{\O}    R(x)\cdot b(t) \dd x = 0. \label{angular_momentum}
\end{align}

 It is standard that the specular reflection boundary condition also preserves the total mass and energy. Therefore, without loss generality, we may assume the initial data always satisfy
\begin{align}
  & \iint_{\O\times \mathbb{R}^3} f_0(x,v) \chi_0 (v)
  \dd x \dd v    = 0, \label{mass_initial} \\
  & \iint_{\O\times \mathbb{R}^3} f_0(x,v)  \chi_4 (v)
  \dd x\dd v = 0, \label{energy_initial}
\end{align}
and hence both total mass and energy are zero:
\begin{align}
  & \iint_{\O\times \mathbb{R}^3} f(t,x,v) \chi_0 (v) \dd x \dd v    = 0, \label{mass_conserve} \\
  & \iint_{\O\times \mathbb{R}^3} f(t,x,v) \chi_4 (v) \dd x \dd v = 0 \label{energy_conserve}.
\end{align}

\subsection{Main Results}
Our main result is an $L^6$-control of the macroscopic part of \eqref{eqn_f}. We rewrite the equation~\eqref{linear_bolzmann} into
\begin{equation}\label{steady_f}
\frac{1}{\e^s}v\cdot \nabla_x f + \frac{1}{\e^{k+s}} \mathcal{L}_\alpha f = g - \p_t f.
\end{equation}

\begin{theorem}\label{thm:l6}For the Boltzmann equations ($\alpha=B$) and Landau equations ($\alpha = L$), let us assume that $f$ solves~\eqref{steady_f} with specular boundary condition~\eqref{specular_f}, $g$ satisfies \eqref{Pg=0}, also assume the initial condition satisfies~\eqref{angular_momentum_initial}, \eqref{mass_initial} and~\eqref{energy_initial}, then we have
\begin{align}
    & \frac{1}{\e^s} \Vert \mathbf{P}f\Vert_{L^6_{x,v}}^6  \lesssim \frac{ \Vert (\mathbf{I}-\mathbf{P})f\Vert_{L^6_{x,v}}^6}{ \e^s} + \frac{ \Vert \mu^{1/4}\mathcal{L}_\alpha f\Vert_{L^2_{x,v}}^6}{ \e^{s+6k}} + \e^{5s}\Vert \nu^{-1/2}(g-\p_t f)\Vert_{L^2_{x,v}}^6. \notag
\end{align}
Recall that $\mathcal{L}_\alpha$ is defined in \eqref{L_alpha}.
\end{theorem}
{

\color{black}The $L^6$ estimate is crucially used in the hydrodynamic limit of the collisional kinetic equation (e.g. \cite{EGKM2,Jang}).
This type of coercivity estimate with integrability gain was first proved by Esposito-Guo-Kim-Marra via a constructive test function approach in the case of diffuse boundary condition \cite{EGKM2}. Their work provided a rigorous derivation of the incompressible Navier-Stokes-Fourier system from the Boltzmann theory in bounded domains. In the proof of Theorem \ref{thm:l6}, we generalize the method of \cite{EGKM2} to the case of specular reflection boundary conditions. We expect our result can be used to study the hydrodynamic limit of the Boltzmann equation in the presence of specular boundary conditions. 

}

{

\color{black}
To prove Theorem \ref{thm:l6}, the key estimate is the $L^6$ control of the momentum $b$. We redesign the test function method for the specular boundary condition inspired by some symmetric Poisson system \cite{Bernou} (see \eqref{system}) and Korn's inequality \cite{DV2} (see Theorem \ref{lemma:estimate_nabla}). We also crucially use an elliptic regularity estimate in $W^{2,\frac{6}{5}}_x$ for the symmetric Poisson system \eqref{system}(apply \eqref{W2p} to \eqref{system_b_5}). 

To prove the $W^{2,\frac{6}{5}}_x$ estimate in \eqref{W2p}, we begin by verifying in Lemma \ref{lemma:complementing} that the symmetric Poisson system and the boundary condition \eqref{system} satisfy the complementing condition in Agmon-Douglis-Nirenberg \cite{ADN}. This result provides an a priori $W^{2,\frac{6}{5}}_x$ estimate for the system \eqref{system}, with an additional $L^{6/5}_x$ estimate that needs to be controlled. Next, in Lemma \ref{lemma:elliptic}, we control the $L^{6/5}_x$ estimate using a classical contradiction argument. We establish existence by constructing an approximating sequence for the source term, denoted as $h_k\in L^2_x \cap L^{6/5}_x$, where $h_k\to h$ in $L^{6/5}_x$. Here $h_k\in L^2_x$ ensures solvability, as stated in Theorem \ref{thm:symmetric_poisson}. The uniqueness of the solution follows from integration by parts. We refer detail to the proof of Theorem \ref{lemma:elliptic}.

}

Our second result is an $L^2$-control of the macroscopic part of \eqref{eqn_f}.
\begin{theorem}\label{thm:coercive} 
For $\alpha=B$ and $\alpha=L$, assume that $f$ solves~\eqref{linear_bolzmann} with specular boundary condition~\eqref{specular_f}, $g$ satisfies \eqref{Pg=0} and initial condition satisfies~\eqref{angular_momentum_initial}, \eqref{mass_initial} and~\eqref{energy_initial}, then we have
\begin{equation*}
\begin{split}
\frac{1}{\e^s}\int_s^t \| \P f (\tau) \|^2_{L^2_{x,v}}  \dd \tau  &\lesssim 
 \frac{1}{\e^s } \int_s^t \Vert \mu^{1/4}(\mathbf{I}-\mathbf{P})f\Vert_{L^2_{x,v}}^2  \dd \tau+ \frac{1}{ \e^{s+2k}} \int_s^t \Vert \mu^{1/4}\mathcal{L}_\alpha f \Vert_{L^2_{x,v}}^2 \dd \tau\\
 &+ G(t) - G(s) + \e^s \int_s^t \Vert g\Vert_{L^2_{x,v}}^2 \dd \tau,
 \end{split}    
\end{equation*}
where $|G(t)|  \lesssim \| f(t) \Vert_{L^2_{x,v}}^2$ is defined in \eqref{G_def}. 
\end{theorem}

{\color{black}
By combining this estimate with Weyl's theorem, Theorem \ref{thm:coercive} leads to a linear full coercivity. This type of coercivity estimate was first proved by Guo (e.g. see \cite{Guo_landau_box}) when there is no boundary. Then a non-constructive proof of the coercivity in the presence of a boundary has been established by Guo and other people (e.g. see \cite{G, kim2018boltzmann, Guo_landau_bdr, Guo_landau_correction}). In the case of diffuse reflection boundary condition, Esposito-Guo-Kim-Marra devised a constructive proof of the coercivity in \cite{EGKM}. Theorem \ref{thm:coercive} can be viewed as an alternative proof of the $L^2$ coercive estimate in \cite{G} in the spirit of the test function method of \cite{EGKM}.

\begin{remark}

In the case of Maxwell boundary conditions, including pure specular boundary conditions, a constructive approach has been proposed by \cite{Bernou}. Their work utilizes Korn's inequality to construct a symmetric Poisson system (see Theorem \ref{thm:symmetric_poisson}) that provides an $L^2$ control of the momentum quantity. They establish the following estimate:
\begin{align*}
\langle \langle v\cdot \nabla_x f + \mathcal{L}_\alpha f,f \rangle \rangle \gtrsim ||| f |||,
\end{align*}
here $\langle \langle \cdot ,\cdot \rangle\rangle$ is a scalar product on $L^2(\O\times \mathbb{R}^3)$, the associated norm $|||\cdot|||$ is equivalent to the standard $L^2$ norm. This result yields the exponential decay of the linear solution $f$.

\end{remark}
}

The following corollary is the consequence of Theorem \ref{thm:coercive} for the Boltzmann equation.
\begin{coro}\label{coro:boltzmann}
We consider the Boltzmann equation where $\mathcal{L}_\alpha$ ($\alpha = B$) is given by the linearized operator $\mathcal{L}_B$ in \eqref{L_alpha} with \eqref{Q_B}. Suppose all assumptions in Theorem \ref{thm:coercive} hold. Then we have 
\begin{equation*}
\begin{split}
\frac{1}{\e^s}\int_s^t \| \P f (\tau) \Vert_{L^2_{x,v}}^2  \dd \tau  \lesssim  G(t) - G(s) +
 \frac{1}{\e^{s+2k} } \int_s^t \Vert (\mathbf{I}-\mathbf{P})f\Vert_{L^2_{x,v}}^2  \dd \tau
 + \e^s \int_s^t \Vert g\Vert_{L^2_{x,v}}^2 \dd \tau  .
\end{split}
\end{equation*}

\end{coro}

Next we consider the case of Landau operator. Following \cite{Guo_landau_box} we define
\begin{equation}\label{sigma}
\sigma^{ij} := \phi^{ij}*\mu = \int_{\mathbb{R}^3} \phi^{ij}(v-v') \mu(v') \dd v'.
\end{equation}
Here $\phi^{ij}$ is defined in~\eqref{coulomb}. Define a weighted norm as
\begin{equation}\label{sigma_norm}
\Vert f\Vert_\sigma^2 :=\int_{\O\times \mathbb{R}^3} \Big[ \sigma^{ij} \p_{v_i} f\p_{v_j} f  +\sigma^{ij} v_i v_j f^2 \Big] \dd x \dd v.
\end{equation}

\begin{coro}\label{coro:landau}
We consider the Landau equation where $\mathcal{L}_\alpha$ ($\alpha=L$) is given by the linearized Landau operator $\mathcal{L}_L$ in \eqref{L_alpha} with \eqref{Q_L}. We have the following $\sigma$-norm control for the macroscopic quantities:
\begin{align*}
 \frac{1}{\e^s} \int_s^t \Vert \mathbf{P}f(\tau)\Vert_{\sigma}^2 \dd \tau  & \lesssim G(t)-G(s) + \frac{1}{\e^{s+2k}} \int_s^t  \Vert (\mathbf{I}-\mathbf{P})f(\tau)\Vert_\sigma^2 \dd \tau +\e^s \int_s^t \Vert g\Vert_{L^2_{x,v}}^2 \dd \tau. 
\end{align*}

\end{coro}

{\color{black}\begin{remark}
Corollary \ref{coro:landau} does not follow directly from Theorem \ref{thm:coercive} since we do not a direct control for $\Vert \mathcal{L}_L f\Vert_{L^2_{x,v}}$. We slightly modify the proof of Theorem \ref{thm:coercive} to handle the contribution of $\mathcal{L}_L f$. We refer detail to the proof at the end of Section \ref{sec:4.1}.
\end{remark}}

\textbf{Outline}. In Section \ref{sec:prelim}, we present several preliminary results regarding the elliptic system and Korn's inequality. Following that, in Section \ref{sec_l6}, we focus on establishing the $L^6$-estimate for the momentum $b$. By combining this estimate with the standard estimates for the mass $a$ and energy $c$, we conclude our main result Theorem \ref{thm:l6}. In Section \ref{sec:l2}, we mimic the proof of Theorem \ref{thm:l6} to obtain the $L^2$ coercive estimate in Theorem \ref{thm:coercive} by a similar test function method.

\hide {\color{black}[Please merge your comments here. You may need to rewrite it:
We will conclude Theorem \ref{thm:coercive} and Theorem \ref{thm:l6} in Section \ref{sec:proof} for general open, bounded and smooth domain. Section \ref{sec:prelim} to Section \ref{sec_l6} serve as preliminary for the proof, where we assume the domain $\O$ is $C^1$, bounded and open.

In fact, Theorem \ref{thm:coercive} and Theorem \ref{thm:l6} also hold for cylindrical domain, i.e, the domain is periodic in the $z$-coordinate. In Section \ref{sec:cylinder}, we will pinpoint the difference conclude both theorems in cylindrical setting.]}
\unhide

\section{Preliminary results
}\label{sec:prelim}

In this section, we summarize some results of Desvillettes-Villani \cite{DV2},   Bernou et al. \cite{Bernou} and Agmon-Douglis-Nirenberg \cite{ADN}, which are crucially used in the proof of the main theorem. 

\subsection{Korn's inequality}
For a smooth vector field $u: \O \rightarrow \R^3$, let $\nabla u$ and $\nablaS u$ represent the Jacobian matrix of $u$ and its symmetric part (a.k.a. the deformation tensor) respectively :
\begin{equation*}
(\nabla u)_{ij} = \frac{\p u_i}{\p x_j}
, \ \  \ \ 
(\nablaS u)_{ij} = \frac{1}{2} \left(
\frac{\p u_i}{\p x_j} + \frac{\p u_j}{\p x_i}
\right).    
\end{equation*}
The antisymmetric part of the Jacobian matrix is denoted by 
\begin{equation*}
(\nablaA u)_{ij} = (\nabla u)_{ij} - (\nablaS u)_{ij}: = \frac{1}{2} \left(
\frac{\p u_i}{\p x_j} - \frac{\p u_j}{\p x_i}
\right).    
\end{equation*}
Later we will use the following identity:
\Be\label{div_Delta}
 \Div(\nablaS u) = {\color{black}\frac{1}{2}} \left(\Delta u + \nabla \Div  u\right).
 \Ee

In this paper we will need the following result of \cite{DV2}:
\begin{theorem}[\cite{DV2}]\label{lemma:estimate_nabla}
Let $\O$ be a $C^1$ bounded, open subset of $\R^3$. Let $u$ be a vector field on $\O$ with $\nabla u \in L^2_x$. Assume that $u$ is tangent to the boundary $\p\O$:
\Be\notag
\forall x \in \p\O, \ \ u (x) \cdot n(x) =0,
\Ee
where $n(x)$ stands for the outer unit normal vector to $\O$ at point $x$. Then there exists a constant $K(\O)>0,$ only depending on $\O$, such that  
\begin{equation*}
    {\color{black} K(\O) \Vert  u\Vert_{H^1_x}^2 } \leq \|\nablaS u 
\|^2_{L^2_x} +  \Big|
 {\color{black} P_\O\Big(  \int_{\O} \nablaA u  \dd x  \Big)}
  \Big|^2.
\end{equation*}
Here, $P_\O$ denotes the orthogonal projection onto the set $\mathcal{A}_\O = \{A\in A_3(\mathbb{R});Ax\in \mathcal{R}_\O\}$. We remind that $A_3(\mathbb{R})$ and $\mathcal{R}_\O$ were defined in \eqref{infinitesimal_rigid}. 
\end{theorem}

\hide
denotes the set of skew-symmetric $3\times 3$-matrices with real coefficients, as
well as the linear manifold of centered infinitesimal rigid displacement fields preserving $\O$:
\begin{align}
\mathcal{R}_\O =  \{R\in \mathcal{R}|b=0, R(x)\cdot n(x) = 0 \text{ for all } x\in \p\O\}.
\end{align}
\unhide

 Due to this result, it is natural to define the following Hilbert spaces:
\begin{equation}
\begin{split}\label{space}
   \mathcal{X} & := \{u:\O\to \mathbb{R}^3: u\in H^1_x, \ u \cdot n(x)=0 \text{ on }\p\O\} ,\\
  \mathcal{X}_0 & :=  \{ u:\O \to \mathbb{R}^3:  u\in H^1_x,  \ u\cdot n(x)=0 \text{ on }\p\O, \ {\color{black} P_\O\Big(  \int_{\O} \nablaA u  \dd x  \Big) = 0}\}.
\end{split}\end{equation}

\hide
\begin{remark}
In the case of unit disk $\O= \{x\in \mathbb{R}^2: |x|<1\}$. 
\[\begin{bmatrix}
    0 & -1 \\
    1 & 0
\end{bmatrix} \in \mathcal{A}_\O.
\]

\end{remark}

This Jacobian matrix is a sum of a symmetric matrix $\nablaS u$ (a.k.a. the deformation tensor) and an antisymmetric matrix $\nablaA u$:
\Be
\nabla u = \nablaS u + \nablaA u,\label{decom_U}
\Ee
where
\Be
\begin{split}
\nablaS u &: = \frac{1}{2} \left[ \nabla u + (\nabla u )^T\right]\\
\nabla^A U &: =  \frac{1}{2} \left[ \nabla U -(\nabla U )^T\right]
\end{split}
\Ee
\unhide

We have the following Poincaré inequality:
\begin{lemma}\label{Lemma: poincare}

For $u\in \mathcal{X}$, we have
\begin{align}
 \Vert u\Vert_{L^2_x}  & \lesssim _\Omega \Vert \nabla u\Vert_{L^2_x}. \label{poincare_ineq}
\end{align}

\end{lemma}

\begin{proof}
Suppose~\eqref{poincare_ineq} is not correct, there exists a sequence $u_k\in \mathcal{X}$ such that
\begin{align}
    &   1 = \Vert u_k \Vert_{L^2_x} \geq k \Vert \nabla u_k\Vert_{L^2_x}. \label{H1_bdd}
\end{align}
By the Rellich-Kondrachov compactness theorem, $H^1_x\subset \subset L^2_x$ and hence, up to the extraction of subsequence, there exists $u_k \to u$ strongly in $L^2_x$. Then for $\phi\in C_c^\infty(\O)$, we have
\begin{align*}
    & \int_{\O} u \p_{x_i}\phi = \lim_{k\to \infty}\int_{\O} u_k \p_{x_i}\phi = -\lim_{k\to \infty} \int_{\O} \p_{x_i} u_k \phi =0,
\end{align*}
where we have used $\Vert \nabla u_k\Vert_{L^2_{x}} \rightarrow 0$ as $k \rightarrow \infty$. 

Thus $\nabla u = 0$ and $u=C$ is a constant. Thus we have $u\in H^1_x$ and by~\eqref{H1_bdd}
\begin{align}
  \Vert u_k - u \Vert_{H^1_x}& = \Vert u_k - u\Vert_{L^2_x} + \Vert \nabla u_k - \nabla u\Vert_{L^2_x} \notag\\
&= \Vert u_k - u\Vert_{L^2_x} + \Vert \nabla u_k \Vert_{L^2_x} \to 0.      \notag
\end{align}
By the trace theorem, we have
\begin{align*}
    &  \Vert u_k - u \Vert_{L^2(\p\O)} \to 0.
\end{align*}
Since $u_k \in \mathcal{X}$ we have $u_k \cdot n(x) |_{\p\O} = 0$ due to the strong convergence in $L^2 (\partial\Omega)$. Then we have $u\cdot n(x)|_{\p\O} = 0$. Since $u=C$, by $n(x)\cdot u = 0$ and~\eqref{n_cover_all} we conclude $u=0$, which contradicts to our initial assumption $\Vert u\Vert_{L^2_x} = 1$.\end{proof}

The following system of elliptic equations with a special boundary condition plays a pivotal role in our proof:
\begin{align}
  -\Div  (\nablaS u   )   = h \ \  &\text{in} \ \ \O , \notag
  \\
u\cdot n = 0  \ \ &\text{on} \ \ \p\O , \label{system}\\
  \nablaS u \,  n  = (\nablaS u : n \otimes n)n  \ \ &\text{on} \ \ \p\O . \notag 
\end{align}
Here, $M:N   := \sum_{ij} m_{ij} n_{ij}$ for $M= (m_{ij})$ and $n= (n_{ij})$.

{\color{black}
\begin{theorem}[Theorem 2.11 in \cite{Bernou}]\label{thm:symmetric_poisson} 

For any given $h\in L^2_x$, there exists a unique solution $u\in \mathcal{X}_0$ to the variational formulation of~\eqref{system}:
\begin{equation}\label{variational}
\int_{\O} \nablaS u : \nablaS v \dd x = \int_{\O} h\cdot v \dd x  \ \text{ for all } \  v\in \mathcal{X}_0.
\end{equation}
Furthermore, suppose $h$ satisfies the compatibility condition: 
\begin{equation}
 \int_{\O} Ax \cdot h(x)  \dd x = 0 \ \text{ for any } \ Ax\in \mathcal{R}_\O. \label{compatibility condition}
\end{equation}
Then the variational solution to~\eqref{variational} satisfies $u\in H^2_x$ with
\begin{equation*}
    \Vert u\Vert_{H^2_x} \lesssim \Vert h\Vert_{L^2_x},
\end{equation*}
and moreover $u$ satisfies~\eqref{system} a.e.

\end{theorem}
{\color{black}
\begin{remark}\label{remark:sym_poisson} To achieve the $H^2_x$ estimate and the second boundary condition in~\eqref{system}, it is crucial to use $  -\Div  (\nablaS u   )   = h $, {\color{black}which is a variant of the elliptic system $ - \Delta u= -\Div  (\nabla u   )   = h$ in \cite{YZ}}. With the first boundary condition in \eqref{system} and the Poincaré inequality from Lemma \ref{Lemma: poincare}, we can establish a unique weak solution using the Lax-Milgram theorem:
\begin{equation*}
\int_{\O} \nabla u :\nabla v \dd x = \int_{\O} h\cdot v\dd x \text{ for any } v\in \mathcal{X}.    
\end{equation*}
However, such elliptic system employed in the test function method is not applicable to the non-flat boundaries, as demonstrated in the proof of Lemma \ref{lemma:l6_b}.

When the domain is axisymmetric \eqref{def:axiss}, we expect a compatibility condition. Such condition arises from Korn's inequality in Theorem \ref{lemma:estimate_nabla}, where we require $|P_\Omega\big(\int_{\O}\nablaA u \dd x\big)| = 0$ to establish a coercive estimate. Consequently, the unique variational solution~\eqref{variational} can be obtained by combining the Lax-Milgram theorem and the Korn's inequality from Theorem \ref{lemma:estimate_nabla}.

For any $v\in \mathcal{X}$ defined in~\eqref{space}, $v - P_\Omega\big(\int_{\mathbb{R}^3}\nablaA v  \dd x\big)x \in \mathcal{X}_0$ . From the additional compatibility condition $\int_{\Omega} h(x)Ax \dd x=0$, substituting $v \to v - P_\Omega\big(\int_{\mathbb{R}^3}\nablaA v  \dd x\big)x$ into \eqref{variational}, the unique solution $u\in \mathcal{X}_0$ further satisfies the following relation:
\begin{align*}
\int_{\Omega} \nablaS u : \nablaS v \dd x = \int_{\Omega} h \cdot v \dd x \text{ for any } v\in \mathcal{X}.
\end{align*}

Moreover, let $h=b$, the momentum, for the axisymmetric domain, we have the preservation of angular momentum~\eqref{angular_momentum_preserve}. We anticipate that this system possesses a unique solution in $H^2_x$ when the angular momentum is zero, which is equivalent to the compatibility condition $\int_{\Omega} Ax\cdot b(x) \dd x = 0$.

\end{remark}

}

\subsection{Elliptic estimate in $W^{2,p}_x$}  We collect preliminary results from Agmon-Douglis-Nirenberg \cite{ADN}.

\begin{lemma}\label{lemma:complementing}
The problem of the elliptic system and boundary conditions of ~\eqref{system} satisfies the complementing condition of Agmon-Douglis-Nirenberg~\cite{ADN}. Thus if $u$ solves \eqref{system} for $h\in L^p_x$ with $1< p< \infty$, by Theorem 10.5 in~\cite{ADN}, we have
\begin{align*}
   \Vert u\Vert_{W^{2,p}_x} &   \lesssim \Vert h\Vert_{L^p_x} + \Vert u\Vert_{L^p_x}.
\end{align*}

\end{lemma}

\begin{proof}
First we express the symmetric Poisson system in~\eqref{system} into a matrix form, which reads
\begin{align*}
    &  \sum_{i,j=1}^3 l_{ij}(\p) u_j = 2h_i,
\end{align*}
where, for real $\xi\in \mathbb{R}^3$ and $\xi\neq 0$, the matrix $l=(l_{ij})$ reads
\begin{align*}
    & l(\xi) = -\begin{bmatrix}
        |\xi|^2 + \xi_1^2 & \xi_1\xi_2 & \xi_1\xi_3 \\
        \xi_1 \xi_2 & |\xi|^2 + \xi_2^2 & \xi_2 \xi_3 \\
        \xi_1 \xi_3 & \xi_2 \xi_3 & |\xi|^2 + \xi_3^2
    \end{bmatrix}.
\end{align*}
Direct computation leads to
\begin{align*}
   D(\xi) & = \det(l(\xi)) = -2|\xi|^6.  
\end{align*}

Let $n$ denote the normal vector and $\Lambda\neq 0$ any tangent vector. Then $D(\Lambda + \tau n)$ has 3 roots in $\tau$:
\begin{align*}
    & \tau_1 = \tau_2 = \tau_3 = i|\Lambda|.
\end{align*}
We denote
\begin{align}
    &   M^+(\Lambda,\tau) = (\tau-i|\Lambda|)^3. \label{M+}
\end{align}

The adjoint matrix of $l(\xi)$ is
\begin{align}
 L(\xi) =   (L_{ij}(\xi)) = &  -  \begin{bmatrix}
        |\xi|^4 + (\xi_2^2+\xi_3^2)|\xi|^2 &  -\xi_1 \xi_2 |\xi|^2  & -\xi_1 \xi_3 |\xi|^2\\
        -\xi_1\xi_2 |\xi|^2 &  |\xi|^4 + (\xi_1^2 +\xi_3^2)|\xi|^2 & -\xi_2\xi_3 |\xi|^2\\
        -\xi_1 \xi_3 |\xi|^2 &  -\xi_2\xi_3 |\xi|^2 & |\xi|^4 + (\xi_1^2 +\xi_2^2 )|\xi|^2
    \end{bmatrix}. \label{adjoint_L}
\end{align}

Next we express boundary operator in~\eqref{system} into a matrix form, which reads
\begin{align*}
  B(\xi)  & = (B_{ij}(\xi)) = \begin{bmatrix}
    \xi\cdot n + \xi_1 n_1 - 2n_1^2 \xi\cdot n  & n_2\xi_1 - 2n_1n_2 \xi \cdot n & n_3\xi_1 - 2n_1n_3\xi \cdot n  \\
    n_1\xi_2 - 2n_2n_1 \xi \cdot n  & \xi\cdot n + \xi_2 n_2 - 2n_2^2 \xi \cdot n & n_3 \xi_2 - 2 n_2 n_3 \xi \cdot n  \\
    n_1\xi_3 - 2n_3n_1 \xi \cdot n  & n_2\xi_3 - 2n_3n_2 \xi \cdot n & \xi \cdot n + \xi_3 n_3 - 2n_3^2 \xi \cdot n  \\
     n_1 &n_2 &n_3 
  \end{bmatrix}   .
\end{align*}
Since $\Lambda \cdot n=0$, we have $(\Lambda + \tau n)\cdot n = \tau$, 
$(\Lambda+\tau n)_i n_j = \Lambda_i n_j + \tau n_i n_j, $ thus
\begin{align*}
    &B(\Lambda + \tau n) = (B_{ij}(\Lambda + \tau n)) = \begin{bmatrix}
        \tau(1-n_1^2) + \Lambda_1 n_1 & \Lambda_1 n_2 -\tau n_1 n_2 &  \Lambda_1 n_3-\tau n_1 n_3 \\
        \Lambda_2 n_1-\tau n_1 n_2& \tau (1-n_2^2) + \Lambda_2 n_2 &  \Lambda_2 n_3-\tau n_2n_3 \\
      \Lambda_3 n_1-\tau n_1n_3 & \Lambda_3 n_2-\tau n_2 n_3 &  \tau (1-n_3^2) + \Lambda_3 n_3\\  
        n_1& n_2 & n_3
    \end{bmatrix}.
\end{align*}
By direct computation, the determinant of the submatrix that consists of first three rows in $B(\Lambda+ \tau n)$ is $0$, thus the rank of the first three rows is $2$.

Without loss of generality, we assume the first two rows are linearly independent, then we claim $n_3\neq 0$. Suppose $n_3= 0$ in such case, we have $n_1^2 + n_2^2 = 1$, $\Lambda_1 n_1 + \Lambda_2 n_2 = 0$. The first two rows become
\begin{align*}
    &  \begin{bmatrix}
        \tau n_2^2 - \Lambda_2 n_2 & \Lambda_1 n_2 - \tau n_1n_2 & 0 \\
        \Lambda_2 n_1 - \tau n_1n_2 &  \tau n_1^2 - \Lambda_1 n_1 & 0 
    \end{bmatrix}.\,
\end{align*}
These two row vectors further become $n_2(\tau n_2 - \Lambda_2, \Lambda_1-\tau n_1)$ and $n_1(\Lambda_2 - \tau n_2,\tau n_1 - \Lambda_1)$, which are linearly dependent. Thus we conclude $n_3\neq 0$ by contradiction. Then we rewrite the matrix $B$ as
\begin{align}
    &\mathfrak{B}(\Lambda + \tau n)  = \begin{bmatrix}
        \tau(1-n_1^2) + \Lambda_1 n_1 & \Lambda_1n_2-\tau n_1 n_2 &  \Lambda_1n_3 -\tau n_1 n_3 \\
        \Lambda_2 n_1-\tau n_1 n_2& \tau (1-n_2^2)+\Lambda_2 n_2 &  \Lambda_2 n_3-\tau n_2n_3 \\
        n_1& n_2 & n_3
    \end{bmatrix}. \label{matrix_bdr}
\end{align}

Next we verify the complementing boundary condition by computing the matrix multiplication with modulo $M^{+}(\Lambda, \tau)$ defined in~\eqref{M+}:
\begin{align}
    & \mathfrak{B}(\Lambda + \tau n) L(\Lambda + \tau n) \mod  \ M^+(\Lambda, \tau)  .\label{matrix_B_times_L}
\end{align}
We first compute $L(\Lambda + \tau n) \mod \  M^+(\Lambda,\tau)$. Note the elements in~\eqref{adjoint_L} consist of $|\xi|^4$ and $\xi_j\xi_k|\xi|^2$. By direct computation, we compute these two terms as
\begin{align}
    & |\Lambda + \tau n|^4 \mod (\tau - i|\Lambda|)^3 = -4 |\Lambda|^2 (\tau - i|\Lambda|)^2,  \label{L_mod_1} 
\end{align}
\begin{align}
    &  (\Lambda_j + \tau n_j)(\Lambda_k + \tau n_k)|\Lambda + \tau n|^2 \mod (\tau - i |\Lambda|)^3 \notag\\
    & = (\tau - i|\Lambda|)\Big[(\Lambda_j \Lambda_k + 3i|\Lambda|\Lambda_j n_k + 3i |\Lambda| \Lambda_k n_j - 5n_jn_k|\Lambda|^2)\tau  \notag\\
    &+ i \Lambda_j \Lambda_k |\Lambda| + n_j\Lambda_k |\Lambda|^2 + \Lambda_j n_k |\Lambda|^2 +3in_jn_k|\Lambda|^3\Big] \notag .
\end{align}
We denote
\begin{equation}\label{L_mod_2}
\begin{split}
   \mathcal{B}_{jk} :=    &  (\Lambda_j \Lambda_k + 3i|\Lambda|\Lambda_j n_k + 3i |\Lambda| \Lambda_k n_j - 5n_jn_k|\Lambda|^2)\tau \\
       &+ i \Lambda_j \Lambda_k |\Lambda| + n_j\Lambda_k |\Lambda|^2 + \Lambda_j n_k |\Lambda|^2 +3in_jn_k|\Lambda|^3.
\end{split}
\end{equation}

Then we have
\begin{align}
    & L(\Lambda + \tau n) \mod \ \  M^+(\Lambda,\tau)   \notag\\
    & = -(\tau-i|\Lambda|)\begin{bmatrix}
      -8|\Lambda|^2(\tau - i|\Lambda|)-\mathcal{B}_{11}  & -\mathcal{B}_{12}  &  -\mathcal{B}_{13}  \\
      -\mathcal{B}_{12}  &  -8|\Lambda|^2(\tau - i|\Lambda|)-\mathcal{B}_{22} & -\mathcal{B}_{23}  \\
       -\mathcal{B}_{13} & -\mathcal{B}_{23}  &-8|\Lambda|^2(\tau - i|\Lambda|)-\mathcal{B}_{33}
    \end{bmatrix}. \label{matrix_l_mod}
\end{align}

To compute the matrix multiplication~\eqref{matrix_B_times_L}, due to the extra $\tau$ factor in~\eqref{matrix_bdr}, we apply~\eqref{L_mod_1} and~\eqref{L_mod_2} to further compute
\begin{align*}
    &\tau |\Lambda + \tau n|^4 \mod (\tau - i|\Lambda|)^3 \\
    &= -4|\Lambda|^2 (\tau - i|\Lambda|)^2 \tau \mod (\tau - i|\Lambda|)^3 = -4i|\Lambda|^3 (\tau - i|\Lambda|)^2.
\end{align*}

\begin{align}
    & \tau (\Lambda_j +\tau n_j)(\Lambda_k + \tau n_k)|\Lambda + \tau n|^2 \mod (\tau - i|\Lambda|^3) \notag \\
    & = (\tau - i |\Lambda|)\tau \mathcal{B}_{jk} \mod (\tau - i|\Lambda|^3) \notag \\
    & = (\tau - i|\Lambda|) \Big[(3i|\Lambda|\Lambda_j \Lambda_k - 5n_j \Lambda_k |\Lambda|^2 - 5\Lambda_j n_k |\Lambda|^2 - 7in_jn_k |\Lambda|^3)\tau \notag \\
    & + \Lambda_j \Lambda_k |\Lambda|^2 + 3i|\Lambda|^3 n_j\Lambda_k + 3i|\Lambda|^3 \Lambda_j n_k -5n_jn_k |\Lambda|^4\Big] \notag .
\end{align}
We denote
\begin{equation*}
\begin{split}
  \mathcal{C}_{jk}:=   &  (3i|\Lambda|\Lambda_j \Lambda_k - 5n_j \Lambda_k |\Lambda|^2 - 5\Lambda_j n_k |\Lambda|^2 - 7in_jn_k |\Lambda|^3)\tau \\
  & + \Lambda_j \Lambda_k |\Lambda|^2 + 3i|\Lambda|^3 n_j\Lambda_k + 3i|\Lambda|^3 \Lambda_j n_k -5n_jn_k |\Lambda|^4.
\end{split}
\end{equation*}

From the above computation, \eqref{matrix_B_times_L} reads
\begin{align*}
    & \mathfrak{B}(\Lambda + \tau n) L(\Lambda + \tau n) \mod (\tau - i|\Lambda|)^3  \\
    &  = \mathfrak{B}(\Lambda+\tau n) \eqref{matrix_l_mod}= -(\tau - i |\Lambda| ) M(\Lambda + \tau n),
\end{align*}
where the elements in the matrix $M=(M_{ij})$ reads
\begin{align*}
  M_{11}  & =  (1-n_1^2)\big[-8i|\Lambda|^3 (\tau - i|\Lambda|) 
 - \mathcal{C}_{11} \big] +n_1n_2 \mathcal{C}_{12} + n_1n_3 \mathcal{C}_{13} \\
 & - 8\Lambda_1 n_1 |\Lambda|^2(\tau - i|\Lambda|) - \Lambda_1 n_1 \mathcal{B}_{11} - \Lambda_1 n_2 \mathcal{B}_{12} - \Lambda_1 n_3 \mathcal{B}_{13} \\
 & = -8i|\Lambda|^3(1-n_1^2) (\tau - i|\Lambda|) + (5n_1 \Lambda_1 |\Lambda|^2 - 3i|\Lambda| \Lambda_1^2)\tau - \Lambda_1^2 |\Lambda|^2 - 3i n_1 \Lambda_1 |\Lambda|^3  \\
 & - 8\Lambda_1 n_1 |\Lambda|^2(\tau - i|\Lambda|) - (3i|\Lambda| \Lambda_1^2 - 5n_1 \Lambda_1 |\Lambda|^2) \tau - \Lambda_1^2 |\Lambda|^2 - 3in_1 \Lambda_1 |\Lambda|^3  \\
 & = -8i|\Lambda|^3(1-n_1^2) (\tau - i|\Lambda|) - (6i|\Lambda|\Lambda_1^2 - 2n_1\Lambda_1 |\Lambda|^2)\tau  - 2\Lambda_1^2 |\Lambda|^2 + 2in_1\Lambda_1 |\Lambda|^3.
\end{align*}

\begin{align*}
   M_{12} &  = -(1-n_1^2)\mathcal{C}_{12} + n_1n_2 \mathcal{C}_{22} + n_1n_3 \mathcal{C}_{23} + 8i |\Lambda|^3 n_1n_2 (\tau - i |\Lambda|) \\
   &-8\Lambda_1n_2 |\Lambda|^2 (\tau - i|\Lambda|) - \Lambda_1 n_1 \mathcal{B}_{12} - \Lambda_1 n_2\mathcal{B}_{22} - \Lambda_1 n_3 \mathcal{B}_{23} \\
   & = 8i|\Lambda|^3 n_1n_2 (\tau - i|\Lambda |)  + (5\Lambda_1 n_2|\Lambda|^2-3i|\Lambda|\Lambda_1 \Lambda_2)\tau - \Lambda_1 \Lambda_2 |\Lambda|^2- 3i |\Lambda|^3 n_2 \Lambda_1 \\
   & -8\Lambda_1n_2 |\Lambda|^2 (\tau - i|\Lambda|) - (3i|\Lambda|\Lambda_1 \Lambda_2 - 5n_2\Lambda_1 |\Lambda|^2)\tau  -\Lambda_1 \Lambda_2 |\Lambda|^2 - 3in_2\Lambda_1 |\Lambda|^3 \\ 
   & = 8i|\Lambda|^3 n_1n_2 (\tau - i|\Lambda |)  -(6i|\Lambda|\Lambda_1\Lambda_2 - 2n_2 \Lambda_1 |\Lambda|^2)\tau - 2\Lambda_1 \Lambda_2 |\Lambda|^2 + 2i |\Lambda|^3 n_2 \Lambda_1        .  
\end{align*}

\begin{align*}
 M_{13}   & = -(1-n_1^2)\mathcal{C}_{13}  + n_1n_2 \mathcal{C}_{23} + n_1n_3 \mathcal{C}_{33}+8i|\Lambda|^3 n_1n_3 (\tau - i|\Lambda|)  \\
 & - 8\Lambda_1 n_3 |\Lambda|^2 (\tau - i|\Lambda|) - n_1\Lambda_1 \mathcal{B}_{13} - \Lambda_1 n_2 \mathcal{B}_{23} - \Lambda_1 n_3 \mathcal{B}_{33}\\
 & = 8i|\Lambda|^3 n_1n_3 (\tau - i|\Lambda|) + (5\Lambda_1 n_3|\Lambda|^2-3i|\Lambda|\Lambda_1 \Lambda_3)\tau - \Lambda_1 \Lambda_3 |\Lambda|^2- 3i |\Lambda|^3 n_3 \Lambda_1 \\
 & - 8\Lambda_1 n_3 |\Lambda|^2 (\tau - i|\Lambda|) - (3i|\Lambda|\Lambda_1 \Lambda_3 - 5n_3 \Lambda_1 |\Lambda|^2)\tau - \Lambda_1 \Lambda_3 |\Lambda|^2 - 3in_3\Lambda_1 |\Lambda|^3 \\
 & = 8i|\Lambda|^3 n_1n_3 (\tau - i|\Lambda|) - (6i|\Lambda|\Lambda_1 \Lambda_3 - 2n_3 \Lambda_1 |\Lambda|^2)\tau - 2\Lambda_1 \Lambda_3 |\Lambda|^2 +2 i |\Lambda|^3 n_3 \Lambda_1 .  
\end{align*}

\begin{align*}
   M_{21} &    =-(1-n_2^2)\mathcal{C}_{12} + n_1n_2 \mathcal{C}_{11} + n_2n_3 \mathcal{C}_{13} + 8i |\Lambda|^3 n_1n_2 (\tau - i |\Lambda|) \\
   & -8\Lambda_2 n_1 |\Lambda|^2 (\tau - i|\Lambda|) - \Lambda_2 n_1 \mathcal{B}_{11} - \Lambda_2 n_2 \mathcal{B}_{12} - \Lambda_2 n_3 \mathcal{B}_{13}\\
   & = 8i|\Lambda|^3 n_1n_2 (\tau - i|\Lambda |)  + (5\Lambda_2 n_1|\Lambda|^2-3i|\Lambda|\Lambda_1 \Lambda_2)\tau - \Lambda_1 \Lambda_2 |\Lambda|^2- 3i |\Lambda|^3 n_1 \Lambda_2 \\
   & - 8\Lambda_2 n_1 |\Lambda|^2 (\tau - i|\Lambda|) - (3i |\Lambda|\Lambda_1 \Lambda_2 - 5n_1\Lambda_2 |\Lambda|^2)\tau - \Lambda_1 \Lambda_2 |\Lambda|^2 - 3in_1\Lambda_2 |\Lambda|^3 \\
   & = 8i|\Lambda|^3 n_1n_2 (\tau - i|\Lambda |) -(6i|\Lambda|\Lambda_1\Lambda_2-2\Lambda_2 n_1 |\Lambda|^2)\tau - 2\Lambda_1 \Lambda_2 |\Lambda|^2 + 2in_1\Lambda_2 |\Lambda|^3.  
\end{align*}

\begin{align*}
  M_{22}  & =  (1-n_2^2)\big[-8i|\Lambda|^3 (\tau - i|\Lambda|) 
 - \mathcal{C}_{22} \big] +n_1n_2 \mathcal{C}_{12} + n_2n_3 \mathcal{C}_{23} \\
 & - 8 \Lambda_2 n_2 |\Lambda|^2 (\tau - i|\Lambda|) - \Lambda_2 n_1\mathcal{B}_{12}-\Lambda_2 n_2 \mathcal{B}_{22} - \Lambda_2 n_3 \mathcal{B}_{23} \\
 & = -8i|\Lambda|^3(1-n_2^2) (\tau - i|\Lambda|) + (5n_2 \Lambda_2 |\Lambda|^2 - 3i|\Lambda| \Lambda_2^2)\tau - \Lambda_2^2 |\Lambda|^2 - 3i n_2 \Lambda_2 |\Lambda|^3 \\
& -8 \Lambda_2 n_2 |\Lambda|^2 (\tau - i|\Lambda|)   -  (3i|\Lambda|\Lambda_2^2 - 5n_2\Lambda_2 |\Lambda|^2)\tau - \Lambda_2^2 |\Lambda|^2 -3in_2\Lambda_2 |\Lambda|^3 \\
& = -8i|\Lambda|^3(1-n_2^2) (\tau - i|\Lambda|) -(6i|\Lambda| \Lambda_2^2-2n_2\Lambda_2 |\Lambda|^2)\tau - 2\Lambda_2^2 |\Lambda|^2 + 2i n_2 \Lambda_2 |\Lambda|^3.
\end{align*}

\begin{align*}
  M_{23}  & = -(1-n_2^2)\mathcal{C}_{23} + n_1n_2\mathcal{C}_{13} + n_2n_3 \mathcal{C}_{33} + 8i |\Lambda|^3 n_2n_3 (\tau - i|\Lambda|) \\
  & -8\Lambda_2 n_3 |\Lambda|^2 (\tau - i|\Lambda|) - \Lambda_2 n_1 \mathcal{B}_{13} - \Lambda_2 n_2 \mathcal{B}_{23} - \Lambda_2 n_3 \mathcal{B}_{33} \\
  & = 8i |\Lambda|^3 n_2n_3 (\tau - i|\Lambda|) + (5\Lambda_2 n_3|\Lambda|^2-3i|\Lambda|\Lambda_3 \Lambda_2)\tau - \Lambda_3 \Lambda_2 |\Lambda|^2- 3i |\Lambda|^3 n_3 \Lambda_2 \\
  &-8\Lambda_2 n_3 |\Lambda|^2 (\tau - i|\Lambda|) - (3i|\Lambda|\Lambda_3 \Lambda_2 - 5n_3 \Lambda_2 |\Lambda|^2)\tau - \Lambda_2 \Lambda_3 |\Lambda|^2 - 3in_3 \Lambda_2 |\Lambda|^3 \\
  & = -8\Lambda_2 n_3 |\Lambda|^2 (\tau - i|\Lambda|) -(6i|\Lambda|\Lambda_3 \Lambda_2 - 2\Lambda_2 n_3 |\Lambda|^2)\tau - 2\Lambda_2 \Lambda_3 |\Lambda|^2 + 2in_3 \Lambda_2 |\Lambda|^3.
\end{align*}

\begin{align*}
  M_{31}  &  = -8n_1|\Lambda|^2 (\tau - i |\Lambda|)-n_1\mathcal{B}_{11} - n_2 \mathcal{B}_{12} - n_3 \mathcal{B}_{13} \\
  & = -8n_1|\Lambda|^2 (\tau - i |\Lambda|) - \tau(3i|\Lambda|\Lambda_1 - 5|\Lambda|^2 n_1) - \Lambda_1 |\Lambda|^2 - 3i|\Lambda|^3 n_1.
\end{align*}

\begin{align*}
  M_{32}  &  = -8n_2|\Lambda|^2 (\tau - i |\Lambda|)-n_1\mathcal{B}_{12} - n_2 \mathcal{B}_{22} - n_3 \mathcal{B}_{23} \\
  & = -8n_2|\Lambda|^2 (\tau - i |\Lambda|) - \tau(3i|\Lambda|\Lambda_2 - 5|\Lambda|^2 n_2) - \Lambda_2 |\Lambda|^2 - 3i|\Lambda|^3 n_2.
\end{align*}

\begin{align*}
  M_{33}  &  = -8n_3|\Lambda|^2 (\tau - i |\Lambda|)-n_1\mathcal{B}_{13} - n_2 \mathcal{B}_{23} - n_3 \mathcal{B}_{33} \\
  & = -8n_3|\Lambda|^2 (\tau - i |\Lambda|) - \tau(3i|\Lambda|\Lambda_3 - 5|\Lambda|^2 n_3) - \Lambda_3 |\Lambda|^2 - 3i|\Lambda|^3 n_3.
\end{align*}

Then we show the rows of~\eqref{matrix_B_times_L} are linearly independent by computing the determinant of $M(\Lambda + \tau n)$:
\begin{align*}
    & \det \big(M(\Lambda +\tau n) \big) \\
    & = M_{31}(M_{12}M_{23} - M_{22}M_{13}) - M_{32}(M_{11}M_{23} - M_{21}M_{13}) + M_{33}(M_{11}M_{22}-M_{12}M_{21}).
\end{align*}

First we compute $M_{31}(M_{12}M_{23} - M_{22}M_{13})$. We expand the subtraction term as
\begin{equation}\label{det:first}
\begin{split}
    & |\Lambda|^2 \Bigg\{ \Big[ 8i|\Lambda|^2 \underbrace{n_1n_2}_{(1)} (\tau - i|\Lambda |)  + (2\underbrace{\Lambda_1 n_2}_{(2)}|\Lambda|-6i\underbrace{\Lambda_1 \Lambda_2}_{(3)})\tau - 2\Lambda_1 \Lambda_2 |\Lambda|+ 2i |\Lambda|^2 n_2 \Lambda_1  \Big] \\
    & \times \Big[8i |\Lambda|^2 \underbrace{n_2n_3}_{(4)} (\tau - i|\Lambda|) + (2\underbrace{\Lambda_2 n_3}_{(5)}|\Lambda|-6i\underbrace{\Lambda_3 \Lambda_2}_{(6)})\tau - 2\Lambda_3 \Lambda_2 |\Lambda|+ 2i |\Lambda|^2 n_3 \Lambda_2  \Big] \\
    & - \Big[-8i|\Lambda|^2 (1-n_2^2) (\tau - i|\Lambda|) + (2n_2 \Lambda_2 |\Lambda| - 6i \Lambda_2^2)\tau - 2\Lambda_2^2 |\Lambda| + 2i n_2 \Lambda_2 |\Lambda|^2  \Big]  \\
    & \ \times \Big[8i|\Lambda|^2 n_1n_3 (\tau - i|\Lambda|) + (2\Lambda_1 n_3|\Lambda|-6i\Lambda_1 \Lambda_3)\tau -2 \Lambda_1 \Lambda_3 |\Lambda|+ 2i |\Lambda|^2 n_3 \Lambda_1 \Big] \Bigg\}.    
\end{split}   
\end{equation}

The common terms of $(1)\times (4)$ combine to be
\begin{align*}
    & n_1n_2 n_2 n_3 + (1-n_2^2)n_1n_3 = n_1n_3.
\end{align*}

The common terms of $(2)\times (5)$ combine to be
\begin{align*}
    &  \Lambda_1 n_2 \Lambda_2 n_3 - n_2 \Lambda_2 \Lambda_1 n_3 = 0. 
\end{align*}

The common terms of $(3)\times (6)$ combine to be
\begin{align*}
    &  \Lambda_1 \Lambda_2 \Lambda_3 \Lambda_2 - \Lambda_2^2 \Lambda_1 \Lambda_3  = 0. 
\end{align*}

The common terms of $(2)\times (6) + (3)\times (5)$ combine to be
\begin{align*}
    &  -\Lambda_1 n_2 \Lambda_3 \Lambda_2 - \Lambda_2 n_3 \Lambda_1 \Lambda_2 + n_2\Lambda_2 \Lambda_1 \Lambda_3 + \Lambda_1 n_3 \Lambda_2^2 = 0 . 
\end{align*}

The common terms of $(1)\times (5) + (2)\times (4)$ combine to be
\begin{align*}
    &  n_1n_2 \Lambda_2 n_3 + \Lambda_1 n_2 n_2 n_3 + (1-n_2^2)\Lambda_1 n_3 - n_1n_3 n_2 \Lambda_2 = \Lambda_1 n_3. 
\end{align*}

The common terms of $(1)\times (6) + (3)\times (4)$ combine to be
\begin{align*}
    &  -n_1n_2 \Lambda_3 \Lambda_2 - \Lambda_1 \Lambda_2 n_2 n_3 - (1-n_2^2)\Lambda_1 \Lambda_3 + n_1n_3  \Lambda_2^2 \\
    &= -n_1n_2 \Lambda_2 \Lambda_3 - \Lambda_1 \Lambda_2 n_2 n_3 + n_2^2 \Lambda_1 \Lambda_3 + n_1n_3 \Lambda_2^2 - \Lambda_1 \Lambda_3. 
\end{align*}

We conclude that
\begin{align*}
  \eqref{det:first}  & = |\Lambda|^2\Big[-64 |\Lambda|^4 n_1n_3(\tau - i |\Lambda|)^2 + 8i|\Lambda|^2 \Lambda_1 n_3 (\tau - i|\Lambda|)[2|\Lambda|\tau  + 2i|\Lambda|^2]  \\
  & +8i|\Lambda|^2 (\tau - i|\Lambda|) [-n_1n_2 \Lambda_2 \Lambda_3 - \Lambda_1 \Lambda_2 n_2 n_3 + n_2^2 \Lambda_1 \Lambda_3 + n_1n_3 \Lambda_2^2 - \Lambda_1 \Lambda_3][6i\tau + 2|\Lambda|]\Big],
\end{align*}
and
\begin{equation}\label{det:1}
\begin{split}
 & M_{31}(M_{12}M_{23} - M_{22}M_{13}) \\
 & = |\Lambda|^3 [-8\underbrace{n_1}_{[1]}|\Lambda| (\tau - i |\Lambda|) - \tau(3i \underbrace{\Lambda_1}_{[2]} - 5|\Lambda| n_1) - \Lambda_1 |\Lambda| - 3i|\Lambda|^2 n_1] \\
 &\times \Big[-64 |\Lambda|^4 \underbrace{n_1n_3}_{[3]}(\tau - i |\Lambda|)^2 + 8i|\Lambda|^2 \underbrace{\Lambda_1 n_3}_{[4]} (\tau - i|\Lambda|)[2|\Lambda|\tau  + 2i|\Lambda|^2] \\
 &+8i|\Lambda|^2 (\tau - i|\Lambda|) \underbrace{[-n_1n_2 \Lambda_2 \Lambda_3 - \Lambda_1 \Lambda_2 n_2 n_3 + n_2^2 \Lambda_1 \Lambda_3 + n_1n_3 \Lambda_2^2 - \Lambda_1 \Lambda_3]}_{[5]}[6i\tau + 2|\Lambda|]\Big].    
\end{split}
\end{equation}

Following a similar computation, we combine common terms to conclude that
\begin{equation}\label{det:2}
\begin{split}
 & -M_{32}(M_{11}M_{23}-M_{21}M_{13})  \\
 & = |\Lambda|^3 [8n_2|\Lambda| (\tau - i |\Lambda|) + \tau(3i\Lambda_2 - 5|\Lambda| n_2) + \Lambda_2 |\Lambda| + 3i|\Lambda|^2 n_2] \\
 &\times \Big[64|\Lambda|^4 n_2n_3 (\tau - i|\Lambda|)^2 - 8i|\Lambda|^2 \Lambda_2 n_3 (\tau - i|\Lambda|)[2|\Lambda|\tau +2i|\Lambda|^2]  \\
 & + 8i|\Lambda|^2 (\tau - i|\Lambda|)[-n_1^2 \Lambda_2 \Lambda_3 - n_2n_3 \Lambda_1^2 + n_1n_2 \Lambda_1 \Lambda_3 + n_1n_3 \Lambda_1\Lambda_2 + \Lambda_2 \Lambda_3][6i\tau + 2|\Lambda|]\Big],    
\end{split}    
\end{equation}

\begin{equation}\label{det:3}
\begin{split}
    & M_{33}(M_{11}M_{22}-M_{12}M_{21})   \\
    & = |\Lambda|^3 [-8n_3|\Lambda| (\tau - i |\Lambda|) - \tau(3i\Lambda_3 - 5|\Lambda| n_3) - \Lambda_3 |\Lambda| - 3i|\Lambda|^2 n_3]  \\
    &\times \Big[-64|\Lambda|^4 (1-n_1^2-n_2^2)(\tau - i|\Lambda|)^2  - 8i|\Lambda|^2(\tau - i|\Lambda|)(n_2\Lambda_2 +n_1\Lambda_1)(2|\Lambda|\tau+2i|\Lambda|^2) \\
    &+8i|\Lambda|^2(\tau - i|\Lambda|)(-n_1^2 \Lambda_2^2 -n_2^2 \Lambda_1^2 + 2n_1n_2\Lambda_1 \Lambda_2 + \Lambda_1^2 + \Lambda_2^2 )(6i\tau + 2|\Lambda|)   \Big].
\end{split}    
\end{equation}

To compute $\det(M(\Lambda + \tau n))$, we combine common terms of $\eqref{det:1}+\eqref{det:2}+\eqref{det:3}$.

The common terms of $[1]\times [3]$ combine to be
\begin{align*}
    &n_1n_1n_3 + n_2n_2n_3 + n_3(1-n_1^2-n_2^2) = n_3.
\end{align*}

The common terms of $[1]\times [4]$ combine to be 
\begin{align*}
    & -n_1\Lambda_1 n_3 - n_2 \Lambda_2 n_3 + n_3 (n_2\Lambda_2 + n_1\Lambda_1)   =0.
\end{align*}

The common terms of $[1]\times [5]$ combine to be 
\begin{align*}
    & n_1^2n_2 \Lambda_2 \Lambda_3 + n_1n_2n_3\Lambda_1\Lambda_2 -n_1n_2^2\Lambda_1\Lambda_3 -n_1^2n_3\Lambda_2^2 + n_1 \Lambda_1\Lambda_3  \\
   & -n_1^2 n_2 \Lambda_2 \Lambda_3 - n_2^2 n_3 \Lambda_1^2 + n_1 n_2^2 \Lambda_1 \Lambda_3 + n_1n_2n_3 \Lambda_1\Lambda_2 + n_2\Lambda_2 \Lambda_3 \\
    & + n_1^2 n_3 \Lambda_2^2 + n_2^2n_3 \Lambda_1^2 - 2n_1n_2n_3 \Lambda_1 \Lambda_2 -n_3\Lambda_1^2 - n_3\Lambda_2^2 \\
    & = n_1\Lambda_1 \Lambda_3 +n_2 \Lambda_2 \Lambda_3 - n_3(|\Lambda|^2 - \Lambda_3^2) = -n_3 |\Lambda|^2 + \Lambda_3 \times (n\cdot \Lambda) = -n_3 |\Lambda|^2.
\end{align*}

The common terms of $[2]\times [3]$ combine to be
\begin{align*}
    &n_1n_3 \Lambda_1 + n_2n_3 \Lambda_2 + (1-n_1^2-n_2^2)\Lambda_3 = n_1n_3 \Lambda_1 + n_2n_3\Lambda_2 +n_3^2\Lambda_3 = n_3 \times (n\cdot \Lambda) = 0.
\end{align*}

The common terms of $[2]\times [4]$ combine to be 
\begin{align*}
    & -\Lambda_1^2 n_3 -  \Lambda_2^2 n_3 + \Lambda_3 (n_2\Lambda_2 + n_1\Lambda_1)   = \Lambda_2 \Lambda_3 n_2 + \Lambda_1 \Lambda_3 n_1 + \Lambda_3^2 n_3 - |\Lambda|^2 n_3 = -n_3|\Lambda|^2 .
\end{align*}

The common terms of $[2]\times [5]$ combine to be 
\begin{align*}
    & n_1n_2 \Lambda_1 \Lambda_2 \Lambda_3 + n_2n_3 \Lambda_1^2 \Lambda_2 - n_2^2 \Lambda_1^2 \Lambda_3 - n_1n_3 \Lambda_1 \Lambda_2^2 +\Lambda_1^2 \Lambda_3 \\
    & - n_1^2 \Lambda_2^2 \Lambda_3 - n_2n_3 \Lambda_1^2 \Lambda_2 + n_1n_2\Lambda_1 \Lambda_2 \Lambda_3 + n_1n_3 \Lambda_1 \Lambda_2^2+ \Lambda_2^2 \Lambda_3 \\
    & +n_1^2 \Lambda_2^2 \Lambda_3 + n_2^2 \Lambda_1^2 \Lambda_3 -2n_1n_2\Lambda_1\Lambda_2\Lambda_3 - \Lambda_1^2 \Lambda_3 - \Lambda_2^2 \Lambda_3  = 0.
\end{align*}

Finally we conclude that
\begin{align*}
  &\det(M(\Lambda+ \tau n))\\
  & = |\Lambda|^3\Big[64 |\Lambda|^5 n_3 (\tau-i|\Lambda|)^2 [8(\tau-i|\Lambda|)-5\tau + 3i|\Lambda|]\\
  & - 8i |\Lambda|^5 n_3 (\tau - i|\Lambda|)(6i\tau + 2|\Lambda|)[8(\tau-i|\Lambda|)-5\tau + 3i |\Lambda|] \\
& -8i|\Lambda|^5 n_3 (\tau - i|\Lambda|)(2\tau + 2i|\Lambda|)(3\tau i + |\Lambda|) \Big] \\
&= 8|\Lambda|^8 n_3 (\tau - i|\Lambda|)\Big[8  (\tau-i|\Lambda|) [3\tau - 5i|\Lambda|]\\
  & + (6\tau -2i |\Lambda|)[3\tau - 5i|\Lambda|] \\
& + (2\tau + 2i|\Lambda|)(3\tau  - i |\Lambda|) \Big]\\
& = 64 |\Lambda|^8 n_3 (\tau - i|\Lambda|)\Big[(\tau - i|\Lambda|)(3\tau - 5i|\Lambda|) + (3\tau - i|\Lambda|)(\tau - i|\Lambda|) \Big] \\
& = 384|\Lambda|^8 n_3 (\tau - i|\Lambda|)^3\neq 0,
\end{align*}
where we have used $|\Lambda|\neq 0$ and $n_3 \neq 0$. This verifies that all three rows of \eqref{matrix_B_times_L} are linearly independent. Thus~\eqref{system} satisfies the complementing boundary condition in~\cite{ADN}, and we conclude the lemma.\end{proof}

\begin{lemma}\label{lemma:elliptic}
When $h\in L^p_x$ for $p=6/5$ satisfies the compatibility condition \eqref{compatibility condition} in Theorem \ref{thm:symmetric_poisson}, there exists a unique solution to ~\eqref{system} such that $u\in W^{2,\frac{6}{5}}_x$ and $P_{\O}\big(\int_{\O}\nabla^a u 
 \dd x\big)=0$. Moreover
\begin{equation}\label{W2p}
\Vert u\Vert_{W^{2,\frac{6}{5}}_x} \lesssim \Vert h\Vert_{L^{6/5}_x}.
\end{equation}
\end{lemma}

\begin{proof}\textit{Proof of Uniqueness.} We prove the uniqueness. Suppose $u_1,u_2\in W^{2,p}_x$ satisfy~\eqref{system} for $h\in L^p_x$ satisfying the compatibility condition \eqref{compatibility condition}, and $P_{\O}\big( \int_{\O}\nabla^a u_1 \dd x\big) =0= P_{\O}\big(\int_{\O}\nabla^a u_2 \dd x\big)$. Then $u=u_1-u_2$ satisfies~\eqref{system} for $h=0$ and $P_{\O}\big(\int_{\O}\nabla^a u \dd x \big) = 0$. Via an integration by parts we compute
\begin{equation}\label{bdr_vanish}
\begin{split}
 0= -\int_{\O} \Div (\nablaS u) u \dd x & = \int_{\O}|\nablaS u|^2 \dd x  -{\color{black}\int_{\p \O}}  \nablaS u n u \dd x   \\
  & = \int_{\O}|\nablaS u|^2 \dd x - {\color{black}\int_{\p \O}} (\nablaS u:n\otimes n) n\cdot u \\
  & = \int_{\O}|\nablaS u|^2 \dd x  ,    
\end{split}
\end{equation}
where we have applied two boundary conditions in~\eqref{system}, consecutively. By the Korn's inequality (Theorem \ref{lemma:estimate_nabla}), with ${\color{black} P_\O\big(\int_{\O} \nablaA u  \dd x\big) = 0}$ we conclude $\Vert u\Vert_{H^1_x} = 0$ and thus $u=0$.

\textit{Proof of Existence.} For $h\in L^p_x, p=\frac{6}{5}$ satisfying the compatibility condition \eqref{compatibility condition}, 
 we choose an approximate sequence $h_k\in L^p_x \cap L^2_x$ such that $\Vert h_k-h\Vert_{L^p_x} = \frac{1}{2^k}$. Then
for all $k$, from $R_i(x),i\in \{1,2\}$ defined in~\eqref{basis} and Definition \ref{def:basis}, we have
\begin{align}
  \Big|\int_{\O} R_i(x) \cdot h_k \dd x \Big|  & \leq \int_{\O}|R_i(x) \cdot (h_k-h)| \dd x \leq C(\O) \Vert h_k-h\Vert_{L^p_x} \leq \frac{C(\O)}{2^k}. \label{R_i_dot_h_k_bdd}
\end{align}
Here, we note that $C(\O)$ does not depend on $k$. 

Now we define 
\begin{equation}
\hat{h}_k := h_k - \sum_{i} \frac{B_{i,k}}{M_i}R_i(x) 
, \label{def:h_k}
\end{equation}where $M_i$ and $B_{i,k}$ are defined as
\begin{align}
M_i     &:= \int_{\O} |R_i(x)|^2 \dd x \gtrsim 1, \label{Mi}\\
B_{i,k} &:= \int_{\O}R_i (x) \cdot h_k (x) \dd x \notag.
\end{align}

 For any $Ax\in \mathcal{R}_\O$, using \eqref{basis}, we can choose $C_1, C_2 \in \mathbb R$ such that $Ax = C_1 R_1 (x) + C_2 R_2 (x)$. Then we compute that 
\begin{align}
     & \int_{\O}Ax\cdot \hat{h}_k \dd x      
    \notag
    \\
    &
   =  \int_{\O} \Big\{  C_1 R_1(x) \cdot h_k + C_2 R_2(x) \cdot h_k- C_1 \frac{B_{1,k}}{M_1} |R_1(x)|^2 - C_2 \frac{B_{2,k}}{M_2} |R_2(x)|^2 \Big\} \dd x \notag\\
    & = C_1 B_{1,k} + C_2B_{2,k} - C_1B_{1,k} -C_2B_{2, k} = 0. \notag
\end{align}

On the other hand, from \eqref{def:h_k}, we derive that 
\begin{align*}
  \Vert \hat{h}_k -h\Vert_{L^p_x}  & \leq \Vert h_k - h\Vert_{L^p_x} + \Vert h_k - \hat{h}_k\Vert_{L^p_x}   \\
  & \leq \frac{1}{2^k} + \Big\Vert \frac{B_{1, k}}{M_1}R_1(x)+\frac{B_{2, k}}{M_2}R_2(x)\Big\Vert_{L^p_x} \leq \frac{1}{2^k} + \frac{C_1(\O)}{2^k}.
\end{align*}
At the last inequality, we have applied~\eqref{R_i_dot_h_k_bdd}. Therefore we conclude that $\Vert \hat{h}_k-h\Vert_{L^p_x} \to 0$ with $\hat{h}_k\in L^2_x\cap L^p_x$ . 

From Theorem \ref{thm:symmetric_poisson}, we can construct a unique solution $u_k\in \mathcal{X}_0$ to \eqref{system} with $h=\hat{h}_k \in L^2_x \cap L^p_x$ of \eqref{def:h_k}. By the a priori estimate in Lemma \ref{lemma:complementing}, we have
\begin{align}
    \Vert u_k\Vert_{W^{2,p}_x} \lesssim \Vert \hat{h}_k\Vert_{L^p_x} +\Vert u_k \Vert_{L^p_x}. \label{aprior}
\end{align}
Now we claim 
\begin{equation}\label{aprior_p}
\Vert u_k\Vert_{L^p_x}\lesssim \Vert \hat{h}_k\Vert_{L^p_x}.   
\end{equation}
Once \eqref{aprior_p} is given, from \eqref{aprior}, we can conclude that 
\begin{equation*}
 \Vert u_k\Vert_{W^{2,p}_x} \lesssim \Vert \hat{h}_k \Vert_{L^p_x} .  
\end{equation*}

Suppose $\Vert u_k\Vert_{p}\lesssim \Vert \hat{h}_k\Vert_{L^p_x}$ is not true, then there exists $v_m\in W^{2,p}_x$ and $g_m \in L^p_x\cap L^2_x$ such that $v_m$ is the unique solution in Theorem \ref{thm:symmetric_poisson} with $h=g_m$, and
\begin{align*}
    & \Vert v_m \Vert_{L^p_x} = 1, \ \ \Vert g_m \Vert_{L^p_x} \to 0.
\end{align*}
Then $\Vert v_m\Vert_{W^{2,p}_x}\lesssim 1$ from~\eqref{aprior}, there exists $v\in W^{2,p}_x$ such that $v_m \rightharpoonup v$ weakly in $W^{2,p}_x$. Moreover, from the compact embedding theorem, we have $\Vert v\Vert_{L^p_x} = 1$.
    
From the Sobolev inequality, we have $\Vert v_m\Vert_{H^1_x}\lesssim \Vert v_m\Vert_{W^{2,\frac{6}{5}}_x}\lesssim 1$, then up to subsequence, we also have $v_m \rightharpoonup v$ weakly in $H^1_x$. Since $v_m$ is the unique solution in Theorem \ref{thm:symmetric_poisson}, we set the weak limit $v$ to be the test function in the weak formulation~\eqref{variational}, then we have
\begin{align}
    & -\int_\O \nablaS v_m : \nablaS v \dd x = \int_\O g_m \cdot v \dd x . \label{var_for_lp}
\end{align}
Note that the above equality holds since $v\in H^1_x$ and $v\cdot n = \lim_{m\to \infty} v_m\cdot n = 0$ on $\partial\Omega$ from the trace theorem. These conditions imply $v\in \mathcal{X}_0$.

From the weak convergence in $H^1_x$, for LHS of~\eqref{var_for_lp} we have
\begin{align*}
    -\int_\O \nablaS v_m :\nablaS v \dd x  \to -\int_\O \nablaS  v :\nablaS v \dd x.
\end{align*}
For RHS of~\eqref{var_for_lp}, by H\"{o}lder inequality, we have
\begin{align*}
    \Big| \int_\O g_m v\dd x \Big| \lesssim \Vert g_m\Vert_{L^{\frac{6}{5}}_x} \Vert v\Vert_{L^6_x} \lesssim \Vert g_m\Vert_{L^{\frac{6}{5}}_x} \Vert v\Vert_{H^1_x} \to 0,
\end{align*}
where we have applied the Sobolev embedding $H^1_x \subset L^{6}_x$. Hence we conclude that $\Vert \nablaS v\Vert_{L^2_x}  = 0$. From the compact embedding, we have $\Vert v_m-v\Vert_{W^{1,\frac{6}{5}}_x}\to 0$, then we obtain $|P_\O \big(\int_{\O} \nablaA v \dd x\big)|=0$ and $v\cdot n=0$ from the trace theorem. By the Korn's inequality in Theorem \ref{lemma:estimate_nabla}, we get {\color{black}$\Vert v\Vert_{H^1_x} = 0$ and thus $v=0$.} This contradicts to $\Vert v\Vert_{L^{\frac{6}{5}}_x}=1$. Then we conclude~\eqref{aprior_p} from the contradiction argument.

Applying the estimate~\eqref{aprior_p} to $u_i-u_j$, we have
\begin{align*}
  \Vert u_i-u_j\Vert_{W^{2,p}_x}  & \lesssim \Vert \hat{h}_i - \hat{h}_j\Vert_{L^p_x} \to 0 \text{ as } i,j\to \infty.  
\end{align*}
This implies $u_k$ is a Cauchy sequence and has a limit $u\in W^{2,p}_x$. Moreover, since $u_k\in \mathcal{X}_0$ is a solution to~\eqref{system} with $h=\hat{h}_k$ and $u_k\in H^2_x$ as in Theorem \ref{thm:symmetric_poisson}, $u_k$ satisfies the boundary condition and thus for all $k$, we have $u_k \cdot n=0$ and $\nablaS u_k n - (\nablaS u_k:n\otimes n)n=0$ on $\partial\Omega$. Firstly, from the trace theorem we have $\Vert u-u_k\Vert_{W^{1,p}(\p\O)} \to 0 $, this leads to $u\cdot n=0$ on $\p\O$, and $\nablaS u n - (\nablaS u:n\otimes n)n=0$. Secondly, from $P_\O \big(\int_{\O}\nablaA u_k \dd x \big) = 0$, for any $k$, we have
\begin{align*}
    &  \Big|P_{\O} \Big(\int_{\O}\nablaA u \dd x \Big) \Big| \leq \Big|P_\O\Big(\int_{\O}\nablaA (u-u_k) \dd x\Big)\Big| \lesssim \Vert u-u_k\Vert_{W^{2,p}_x} \lesssim_\O \frac{1}{2^k},
\end{align*}
we conclude that $P_\O\big(\int_{\O}\nablaA u \dd x\big) = 0$.

Taking $k\to \infty$ for both $u_k$ and $\hat{h}_k$ we conclude that
\begin{align*}
    & -\Div(\nablaS u) = h.
\end{align*}
Thus $u\in W^{2,p}_x$ is a solution to~\eqref{system} with $h\in L^p_x$ and $P_\O\big(\int_{\O}\nablaA u \dd x\big)=0$. Moreover, such solution satisfies~\eqref{W2p} with $\Vert \hat{h}_k-h\Vert_{L^p_x} \to 0 $. \end{proof}

\section{$L^6$ estimate and proof of Theorem \ref{thm:l6}}\label{sec_l6}

In this section we will conclude the Theorem \ref{thm:l6}. A key estimate is an $L^6$ control for the momentum $b$ in Lemma \ref{lemma:l6_b}. To complete the proof of Theorem \ref{thm:l6}, we also need the estimate of mass $a$ and energy $c$, we present them in Lemma \ref{lemma:l6_a_c}.

\begin{lemma}\label{lemma:l6_b}
Suppose all assumptions in Theorem \ref{thm:l6} hold, then we have
\begin{equation*}
\frac{1}{\e^s} \Vert b\Vert_{L^6_x}^6    \lesssim  \frac{\Vert (\mathbf{I}-\mathbf{P})f\Vert_{L^6_{x,v}}^6}{\e^s} + \frac{ \Vert \mu^{1/4}\mathcal{L}_\alpha f\Vert_{L^2_{x,v}}^6}{\e^{s+6k}} + \e^{5s} \Vert \nu^{-1/2}(g-\p_t f)\Vert_{L^2_{x,v}}^6.    
\end{equation*}

\end{lemma}
\begin{remark}
In the proof, we introduce a new test function, denoted as \eqref{test_b_1}, which includes a factor $\phi_b$ satisfying the symmetric Poisson system \eqref{system_b_5}. By considering the action of the transport operator on this test function, we obtain the symmetric Poisson operator $\Div(\nablaS u)$. The detailed computation is provided in \eqref{transport_b}.

Importantly, with the boundary condition specified in the symmetric Poisson system \eqref{system_b_5}, the boundary contribution in the weak formulation \eqref{steady_weak_formulation} vanishes due to the combination of terms $K_{12}$ and $K_{13}$ in \eqref{bdr_2_3_vanish}. This key observation circumvents an issue mentioned in Remark \ref{remark:sym_poisson}, where for a non-flat boundary, the term $K_{13}$ does not vanish when using the boundary condition $u\cdot n=0$ and $\partial_n u\cdot \tau = 0$ ($\tau$ denotes the tangential vector).

\end{remark}

Here, we emphasize that Lemma \ref{lemma:l6_b} holds for both non-axisymmetric and axisymmetric domains. To better understand the contributions of axisymmetry, we provide separate proofs of this lemma for these two settings in Section \ref{sec:nonaxis} and Section \ref{sec:axis}, respectively.

Then, in Section \ref{sec:l6_a_c}, we provide the estimate of $a$ and $c$. At the end of the section, we collect the results from Lemma \ref{lemma:l6_b} and Lemma \ref{lemma:l6_a_c} to conclude Theorem \ref{thm:l6}.

\subsection{Proof of Lemma \ref{lemma:l6_b} under non-axisymmetric domain}\label{sec:nonaxis}

We first consider a non-axisymmetric domain, so that $\mathcal{R}_\O = \{0\}$. Then in Theorem \ref{thm:symmetric_poisson}, $\mathcal{X}_0 = \mathcal{X}$ and there is no compatibility condition for $h$ since $Ax = 0$ for $Ax \in \mathcal{R}_\O$.

For the proof we will use special test functions together with the weak formulation of~\eqref{steady_f}:
\begin{align}
&  -\frac{1}{\e^s} \int_{\O\times \mathbb{R}^3} v\cdot \nabla_x \psi f \dd x \dd v    \notag  \\
 & = \int_{\gamma} \psi f \dd \gamma + \frac{1}{\e^{s+k}} \int_{\O\times \mathbb{R}^3} \mathcal{L}_\alpha f \psi \dd x \dd v + \int_{\O\times \mathbb{R}^3} (g-\p_t f)\psi \dd x \dd v : = K_1+K_2+K_3. \label{steady_weak_formulation}
\end{align}
{\color{black} Here $\gamma$ is defined in \eqref{gamma_bdr}, $\dd \gamma:= (n(x)\cdot v) \dd S_x \dd v$, and $\dd S_x$ is the surface integral on the boundary.}

We define the following Burnette function of the space $\mathcal{N}^\perp$ as
\begin{equation}\label{B_ij}
\hat A_{ij}(v) :=  \left(v_i v_j - \frac{\delta_{ij}}{3}|v|^2\right) \sqrt \mu  \text{ for } i,j=1,2,3.
\end{equation}

A key property of~\eqref{B_ij} is
\begin{equation}\label{B_ij_property}
\mathbf{P}(\hat{A}_{ij}) = 0.
\end{equation}

We will use the following basic facts for the velocity integral:
\begin{align*}
  \int_{-\infty}^{\infty} e^{-v^2/2} \dd v &   = \sqrt{2\pi},\\
  \int_{-\infty}^{\infty} v^2 e^{-v^2/2} \dd v & = \sqrt{2\pi} , \\
  \int_{-\infty}^{\infty} v^4 e^{-v^2/2} \dd v & = 3\sqrt{2\pi}.
\end{align*}
The above computation leads to the following facts for the integral of $\mu(v)$ in~\eqref{Maxwellian}: for $i=1,2,3$
\begin{equation}\label{fact}
\begin{split}
     \int_{\mathbb{R}^3}   \mu(v) \dd v & = 1 , \\
    \int_{\mathbb{R}^3}  |v_i|^2  \mu(v) \dd v & = 1, \\
    \int_{\mathbb{R}^3} |v_i|^2 |v_j|^2 \mu(v)\dd v & =1 \ \text{ for } i\neq j, \\
    \int_{\mathbb{R}^3} |v_i|^4 \mu(v) \dd v & = 3. 
\end{split}
\end{equation}

Let $\phi_b$ be the solution of the following system:
\begin{align}
  -\Div(\nablaS \phi_b)  & = \frac{1}{2} b^5 \text{ in }\O \notag \\
  \phi_b \cdot n & = 0  \text{ on } \p\O\label{system_b_5}\\
  \nablaS \phi_b n& = (\nablaS \phi_b : n\otimes n)n \text{ on } \p\O .\notag
\end{align}
Here $b^5 = (b_1^5,b_2^5,b_3^5)$. From elliptic estimate in Lemma \ref{lemma:elliptic}, we have 
{\color{black}
\begin{align}
 \Vert \phi_b\Vert_{W^{2,\frac{6}{5}}_x} \lesssim \Vert b^5\Vert_{L^{\frac{6}{5}}_x} \lesssim \Vert b\Vert_{L^6_x}^5. \label{phi_b_W_2}
\end{align}
}

We choose test function as
\begin{align}
  &\psi=\psi_b \notag\\
  &:= \sum_{i,j=1}^3  \p_j \phi^i_b v_i v_j \mu^{1/2} - \sum_{i=1}^3  \p_i \phi_b^i  \mu^{1/2} \label{test_b_1}\\
  & = \sum_{i,j=1}^3 \p_j \phi^i_b \mu^{1/2} \Big[v_i v_j - \frac{\delta_{ij}}{3}|v|^2\Big] + \sum_{i,j=1}^3 \p_j \phi^i_b \mu^{1/2} \frac{\delta_{ij}}{3}|v|^2 - \sum_{i=1}^3 \p_i \phi_b^i \mu^{1/2} \notag\\
  & = \sum_{i,j=1}^3 \p_j \phi_b^i  \hat{A}_{ij} + \sum_{i=1}^3 \p_i \phi^i_b \mu^{1/2} \Big[\frac{|v|^2-3}{3} \Big] \notag\\
  & = \sum_{i,j=1}^3 \p_j \phi_b^i  \hat{A}_{ij} + \sum_{i=1}^3 \p_i \phi_b^i  \chi_4 \frac{\sqrt{6}}{3}. \label{test_b_2}
\end{align}
Here $\chi_4$ and $\hat{A}_{ij}$ are defined in~\eqref{basis_chi} and~\eqref{B_ij} respectively.

We compute the transport operator $-v\cdot \nabla_x$ on $\psi_b$ using~\eqref{test_b_1}:
\begin{align}
    & -v\cdot \nabla_x \psi_b = -\sum_{i,j,k=1}^3 \p_{kj}\phi_b^i v_i v_j v_k \mu^{1/2} + \sum_{i,k=1}^3 v_k\p_{ki} \phi_b^i \mu^{1/2}  \notag \\
    & = -\sum_{i,j,k=1}^3 \p_{kj} \phi_b^i (\mathbf{I}-\mathbf{P})(v_i v_j v_k \mu^{1/2}) - \sum_{i,j,k=1}^3  \p_{kj} \phi_b^i  \mathbf{P}(v_i v_j v_k \mu^{1/2}) + \sum_{i,k=1}^3 v_k\p_{ki}\phi_b^i \mu^{1/2}. \label{test_b_derivative}
\end{align}
For $\mathbf{P}(v_iv_j v_k \mu^{1/2})$, when $i=j=k$, we have
\begin{align*}
   \mathbf{P}((v_i)^3 \mu^{1/2}) & = \langle (v_i)^3 \mu^{1/2}, \chi_i \rangle \chi_i = 3\chi_i,
\end{align*}
where we applied the computation in~\eqref{fact}.

When $i=j\neq k$, we apply~\eqref{fact} to have
\begin{align*}
   \mathbf{P}((v_i)^2 v_k \mu^{1/2}) & = \langle (v_i)^2 v_k \mu^{1/2}, \chi_k \rangle \chi_k = \chi_k.
\end{align*}
When $j=k\neq i$, we apply~\eqref{fact} to have
\begin{align*}
   \mathbf{P}((v_j)^2 v_i \mu^{1/2}) & = \langle (v_j)^2 v_i \mu^{1/2}, \chi_i \rangle \chi_i = \chi_i.
\end{align*}
When $i=k\neq j$, we apply~\eqref{fact} to have
\begin{align*}
   \mathbf{P}((v_i)^2 v_j \mu^{1/2}) & = \langle (v_i)^2 v_j \mu^{1/2}, \chi_j \rangle \chi_j = \chi_j.
\end{align*}
The above computation yields
\begin{align}
   &-\sum_{i,j,k=1}^3 \p_{kj}\phi^i_b \mathbf{P}(v_iv_jv_k \mu^{1/2}) \notag\\ 
   &=-3\sum_{ {\color{black}i=j=k}=1}^3 \p_{ii}\phi_b^i \chi_i-\sum_{j=k\neq i} \p_{jj}\phi^i_b \chi_i \notag\\
   &-\sum_{i=j\neq k} \p_{ki}\phi^i_b \chi_k -\sum_{i=k\neq j} \p_{ij}\phi^i_b \chi_j \notag\\
   &=-3\sum_{i=1}^3 \p_{ii}\phi_b^i \chi_i-\sum_{j\neq i} \p_{jj}\phi^i_b \chi_i \notag\\
   &-\sum_{i\neq k} \p_{ki}\phi^i_b \chi_k -\sum_{i\neq j} \p_{ij}\phi^i_b \chi_j .  \label{estimate_P}
\end{align}
Here we note that the last two terms are the same.

The last term in~\eqref{test_b_derivative} reads
\begin{align}
    &  \sum_{i=k}^3 \p_{ii} \phi_b^i  \chi_i + \sum_{i\neq k}^3 \p_{ki}\phi_b^i  \chi_k. \label{estimate_second}
\end{align}

Collecting~\eqref{estimate_P} and~\eqref{estimate_second}, the last two terms in~\eqref{test_b_derivative} combine to be
\begin{align}
    &- \sum_{i,j,k=1}^3  \p_{kj} \phi_b^i  \mathbf{P}(v_i v_j v_k \mu^{1/2}) + \sum_{i,k=1}^3 v_k\p_{ki}\phi_b^i \mu^{1/2} \notag\\
    & = -2\sum_{i=1}^3 \p_{ii} \phi^i_b 
    \chi_i
    - \sum_{j\neq i} \p_{jj} \phi^i_b  \chi_i - \sum_{i\neq j}\p_{ij}\phi_b^i \chi_j \notag \\
    & = -\sum_{i=1}^3 \p_{ii} \phi^i_b \chi_i - \sum_{j\neq i} \p_{jj} \phi^i_b  \chi_i \notag\\ 
    & -\sum_{i=1}^3 \p_{ii} \phi^i_b  \chi_i- \sum_{{\color{black}j}\neq {\color{black}i}}\p_{{\color{black}ji}}\phi_b^{{\color{black}j}}\chi_{{\color{black}i}} \notag\\
    & =-\sum_{i=1}^3 \chi_i  \sum_{j=1}^3 \p_{jj}\phi_b^i    - \sum_{i=1}^3 \chi_i  \sum_{j=1}^3 \p_{ij}\phi_b^j \notag\\
    & = -\sum_{i=1}^3  \chi_i [\Delta \phi^i_b + \p_i \Div(\phi_b)] = \sum_{i=1}^3 \chi_i b_i^5.  \label{transport_b}
\end{align}
In the last line we used that $\phi_b^i$ is the solution to the system~\eqref{system_b_5}, and $\Div(\nablaS \phi_b) = \frac{1}{2} (\Delta \phi_b + \nabla \Div \phi_b)$ from~\eqref{div_Delta}.

Then for~\eqref{test_b_derivative} we conclude that
\begin{align*}
   \eqref{test_b_derivative} & = \sum_{i=1}^3  \chi_i b_i^5 -\sum_{i,j,k=1}^3 \p_{kj} \phi_b^i  (\mathbf{I}-\mathbf{P})(v_i v_j v_k \mu^{1/2}).
\end{align*}

Then LHS of~\eqref{steady_weak_formulation} becomes
\begin{align}
   LHS & =  \frac{1}{\e^s} \Big\{ \int_{\O} |b|^6 \dd x   -\underbrace{\sum_{i,j,k=1}^3  \int_\O \p_{kj}\phi_b^i \langle v_iv_jv_k \mu^{1/2}, (\mathbf{I}-\mathbf{P})f \rangle  \dd x  }_{E_1}  \Big\}, \label{lhs_l6_non}
\end{align}
where, by Young's inequality with $\frac{1}{6}+ \frac{5}{6} = 1$,
\begin{align*}
    |E_1| &\lesssim o(1) \Vert \nabla^2 \phi_b\Vert_{L^{\frac{6}{5}}_x}^{\frac{6}{5}} +  \Vert (\mathbf{I}-\mathbf{P})f\Vert_{L^6_{x,v}}^6 \\
    &\lesssim o(1) \Vert \phi_b\Vert_{W^{2,\frac{6}{5}}_x}^{\frac{6}{5}} + \Vert (\mathbf{I}-\mathbf{P})f\Vert_{L^6_{x,v}}^6 \lesssim o(1) \Vert b\Vert_{L^6_x}^6 + \Vert (\mathbf{I}-\mathbf{P})f\Vert_{L^6_{x,v}}^6.
\end{align*}
Here we have used \eqref{phi_b_W_2}.

Next we estimate $K_i, 1\leq i\leq 3$ in~\eqref{steady_weak_formulation}. For $K_1$ we use the representation of $\psi_b$ in~\eqref{test_b_1} to have
\begin{align*}
 \int_\gamma \psi_b f \dd \gamma   = &   \int_{\p\O\times \mathbb{R}^3} (n\cdot v)\Big(\sum_{i,j=1}^3  \p_j \phi^i_b v_i v_j \mu^{1/2} - \sum_{i=1}^3  \p_i \phi_b^i  \mu^{1/2}\Big)f\dd v \dd S_x.
\end{align*}
The second term vanishes by applying the specular boundary condition of $f$. For the first term, we first compute the velocity integral. Through a coordinate rotation, we may assume that $n=(1,0,0)$. Then we compute the $\dd v$ integral as
\begin{align*}
& \sum_{i,j=1}^3  \int_{\mathbb{R}^3} v_1 \p_j \phi_b^i v_i v_j \mu^{1/2} f \dd v      \\
    &=  \p_1 \phi_b^1 \int_{\mathbb{R}^3} (v_1)^3 \mu^{1/2} f \dd v +  \sum_{i=2}^3 \p_1 \phi_b^i  \int_{\mathbb{R}^3} (v_1)^2 v_i \mu^{1/2} f \dd v \\
    & + \sum_{j=2}^3 \p_j \phi_b^1 \int_{\mathbb{R}^3} (v_1)^2 v_j \mu^{1/2} f\dd v +  \sum_{i,j=2}^3  \p_j \phi_b^i \int_{\mathbb{R}^3}  v_1 v_i v_j \mu^{1/2} f \dd v\\
    & := K_{11} + K_{12} + K_{13} + K_{14}.
\end{align*}
From $n=(1,0,0)$ we have specular reflection $f(v_1,v_2,v_3) = f(-v_1,v_2,v_3)$, this leads to $K_{11}=K_{14} = 0$. For $K_{12}$ and $K_{13}$, we combine them to have
\begin{align}
  K_{12}+K_{13}  &  =  \int_{\mathbb{R}^3} (v_1)^2 \Big[ v_2 (\p_2 \phi_b^1 + \p_1 \phi_b^2) + v_3 (\p_3 \phi_b^1 + \p_1 \phi_b^3) \Big] \mu^{1/2} f \dd v  . \label{bdr_2_3_vanish}
\end{align}

Then we apply the second boundary condition in~\eqref{system_b_5} to have
\begin{align*}
     \begin{bmatrix}
       \p_{1}\phi_b^1 \\
        \frac{\p_1 \phi_b^2 + \p_2 \phi_b^1}{2} \\
        \frac{\p_1 \phi_b^3 + \p_3 \phi_b^1}{2}
    \end{bmatrix} = \begin{bmatrix}
        \p_1 \phi_b^1 \\
        0 \\
        0 
    \end{bmatrix}.
\end{align*}
Here we note that $\p_j \phi_b^1 \neq 0, j=2,3$ and thus $K_{13}$ does not vanish in general for a non-flat boundary. We conclude that
\begin{align*}
    K_{12} + K_{13} = 0.
\end{align*}
Hence we conclude the estimate for $K_1$ as
\begin{align}
    K_1 = 0. \label{K_1_estimate}
\end{align}

For $K_2$ and $K_3$, we apply Sobolev-Gagliardo-Nirenberg inequality to have
\begin{align*}
    & \Vert \nabla \phi_b\Vert_{L^2_x} \lesssim \Vert \nabla \phi_b\Vert_{W^{1,\frac{6}{5}}_x} \lesssim \Vert b\Vert_{L^6_x}^5.
\end{align*}

Thus $K_2$ and $K_3$ are bounded as
\begin{align}
  |K_2|  & \lesssim \frac{1}{\e^{s+k}}\Vert \mu^{1/4}\mathcal{L}_\alpha f \Vert_{L^2_{x,v}} \Vert b\Vert_{L^6_x}^5 \lesssim \frac{o(1)}{\e^s} \Vert b\Vert_{L^6_x}^6 + \frac{1}{ \e^{s+6k}} \Vert \mu^{1/4}\mathcal{L}_\alpha f\Vert^6_{L^2_{x,v}}, \label{K2_b_l6_bdd}
\end{align}
\begin{align}
  |K_3|  & \lesssim \Vert \mu^{1/4} (g-\p_t f)\Vert_{L^2_{x,v}} \Vert b\Vert_{L^6_x}^5 \lesssim \frac{o(1)}{\e^s} \Vert b\Vert_{L^6_x}^6 + \e^{5s} \Vert \nu^{-1/2}(g-\p_t f)\Vert^6_{L^2_{x,v}}  . \label{K3_b_l6_bdd}
\end{align}

Collecting \eqref{lhs_l6_non}, \eqref{K_1_estimate}, \eqref{K2_b_l6_bdd} and \eqref{K3_b_l6_bdd}, we conclude Lemma \ref{lemma:l6_b} for the non-axisymmetric case.

\subsection{Proof of Lemma \ref{lemma:l6_b} under axisymmetric domain}\label{sec:axis}
We consider the case that the domain is axisymmetric. Recall \eqref{basis} and Definition \ref{def:basis}. \hide For $\dim(\O)=3$, we take $R_1(x)$ and $R_2(x)$ to be an orthonormal basis of $\mathcal{R}_\O$, so that for any $R(x)\in \mathcal{R}_{\O}$, we can express 
\begin{align}
    & R(x) = C_1 R_1(x) + C_2R_2(x) \text{ for some } C_1,C_2, \ R_1(x)\cdot R_2(x) = 0.     \label{basis}
\end{align}
Here note that if $\dim(\mathcal{R}_\O)=1$, then $R_1(x) = R_2(x)$.\unhide

Then we denote 
\begin{align}
  B_1  &  := \int_{\O}R_1(x)\cdot b^5(x)\dd x  , \notag\\
  B_2 & := \int_{\O}R_2(x)\cdot b^5(x) \dd x. \notag
\end{align}
Recall $M_i$ in \eqref{Mi}.

For a given source $b^5(x)$ in~\eqref{system_b_5}, we denote $h(x)$ as the subtraction of the source and basis~\eqref{basis}:
\begin{equation}\label{source_l6}
h(x) = b^5(x) - \frac{B_1}{M_1}R_1(x) - \frac{B_2}{M_2}R_2(x).    
\end{equation}
For any $Ax\in \mathcal{R}_\O$, we have that for some $C_1,C_2$, $Ax = C_1R_1(x) + C_2R_2(x)$, then
\begin{align}
    &\int_{\O} Ax \cdot h(x) \dd x     \notag\\
    &  = \int_{\O} \Big\{  C_1 R_1(x) \cdot b^5 + C_2 R_2(x) \cdot b^5 - C_1 \frac{B_1}{M_1} |R_1(x)|^2 - C_2 \frac{B_2}{M_2} |R_2(x)|^2 \Big\} \dd x \notag\\
    & = C_1 B_1 + C_2B_2 - C_1B_1 -C_2B_2 = 0. \notag
\end{align}
Thus~\eqref{source_l6} satisfies the compatibility condition in Theorem \ref{thm:symmetric_poisson}.

Then we consider a variant of~\eqref{system_b_5} as following:
\begin{align}
  -\Div(\nablaS \phi_b)  & = \frac{1}{2} \times \eqref{source_l6}  \text{ in }\O \notag \\
  \phi_b \cdot n & = 0  \text{ on } \p\O\label{system_b_5_sym}\\
  \nablaS \phi_b n& = (\nablaS \phi_b : n\otimes n)n \text{ on } \p\O .\notag
\end{align}
From elliptic estimate in Lemma \ref{lemma:elliptic}, there exists a unique $\phi_b$ to~\eqref{system_b_5_sym} that satisfies
{\color{black}
\begin{align}
   & \Vert \phi_b\Vert_{W^{2,\frac{6}{5}}_x} \lesssim \Big \Vert b^5 - \frac{B_1}{M_1}R_1(x) - \frac{B_2}{M_2}R_2(x) \Big\Vert_{L^{\frac{6}{5}}_x} \notag\\
    & \lesssim \Vert b\Vert_{L^6_x}^5 + |B_1| + |B_2| \lesssim \Vert b\Vert_{L^6_x}^5,   \label{b_W2_sym}
\end{align}
}
where we have used $\Vert R_i\Vert_{L^\infty_x} \lesssim 1$ for a bounded $\O$ and used H\"{o}lder inequality to have
\begin{align*}
   |B_i| & \lesssim \int_{\O} |b^5(x)| \dd x   \lesssim_\O \Vert b\Vert_{L^6_x}^5.
\end{align*}

Following the computation in~\eqref{lhs_l6_non}, LHS of~\eqref{steady_weak_formulation} becomes
\begin{align*}
  LHS  & = \frac{1}{\e^s} \Big\{\int_{\O}b(x) \cdot \big[ b^5(x) -  \frac{B_1}{M_1}R_1(x)-\frac{B_2}{M_1}R_2(x)\big]  \\
  & \ \ \ \ -\underbrace{\sum_{i,j,k=1}^3  \int_\O \p_{kj}\phi_b^i \langle v_iv_jv_k \mu^{1/2}, (\mathbf{I}-\mathbf{P})f \rangle  \dd x}_{E_2}  \Big\}    \\
  & = \frac{1}{\e^s} \Big\{\int_{\O} |b|^6 \dd x -E_2 \Big\},
\end{align*}
where we have used~\eqref{angular_momentum} to have
\begin{align*}
    \int_{\O} b(x)\cdot R_i(x) \dd x = 0.
\end{align*}

Since the $W^{2,\frac{6}{5}}_x$ estimate in~\eqref{b_W2_sym}  has the same upper bound as~\eqref{phi_b_W_2}, the estimates for $E_2$ and $K_1,K_2,K_3$ on the RHS of~\eqref{steady_weak_formulation} are also the same as the non-axisymmetric case. Thus we conclude Lemma \ref{lemma:l6_b} for the axisymmetric case.

\subsection{$L^6$ estimate of $a$ and $c$}\label{sec:l6_a_c}

\begin{lemma}\label{lemma:l6_a_c}
\begin{align}
  \frac{1}{\e^s} \Vert a\Vert_{L^6_x}^6  &  \lesssim  \frac{\Vert (\mathbf{I}-\mathbf{P})f\Vert_{L^6_{x,v}}^6}{\e^s} + \frac{\Vert \mu^{1/4}\mathcal{L}_\alpha f\Vert_{L^2_{x,v}}^6}{\e^{s+6k}} + \e^{5s} \Vert \nu^{-1/2} (g - \p_t f)\Vert_{L^2_{x,v}}^6. \label{a_l6}
\end{align}

\begin{align}
  \frac{1}{\e^s} \Vert c\Vert_{L^6_x}^6  &  \lesssim  \frac{\Vert (\mathbf{I}-\mathbf{P})f\Vert_{L^6_{x,v}}^6}{\e^s} + \frac{\Vert \mu^{1/4}\mathcal{L}_\alpha f\Vert_{L^2_{x,v}}^6}{\e^{s+6k}} + \e^{5s} \Vert \nu^{-1/2}(g-\p_t f)\Vert_{L^2_{x,v}}^6. \label{c_l6}
\end{align}
    
\end{lemma}

\begin{proof}We follow the proof of \cite{guo2018boltzmann}. Again we will use the weak formulation~\eqref{steady_weak_formulation} in our estimate.

In the proof we will use another Burnette function of the space $\mathcal{N}^\perp$:
\begin{equation*}
\hat{B}_{i} :=  v_i \frac{|v|^2-5}{\sqrt{10}} \mu^{1/2} \text{ for } i = 1,2,3.
\end{equation*}

\textit{Proof of~\eqref{a_l6}}.
We choose $\phi_a$ to be the solution of the following problem:
\begin{align*}
  -\Delta \phi_a   = a^5 - \frac{1}{|\O|} \int a^5 \dd x , \ \
   \frac{\p \phi_a}{\p n} = 0 \text{ on } \p \O, \ \ \int \phi_a \dd x = 0.
   \end{align*}
In~\eqref{steady_weak_formulation}, we take the test function to be
\begin{align}
  \psi = \psi_a   &  := \sum_{i=1}^3  \p_i \phi_a v_i(|v|^2 -10) \mu^{1/2}   = \sum_{i=1}^3 \p_i \phi_a  (\sqrt{10}\hat{B}_i - 5\chi_i). \label{test_a}
\end{align}
Standard elliptic estimate leads to
\begin{align*}
  \Vert \phi_a \Vert_{W^{2,\frac{6}{5}}_x}  & \lesssim   \Big\Vert a^5 - \frac{1}{|\O|}\int a^5 \dd x \Big\Vert_{L^{\frac{6}{5}}_x} \lesssim_\O \Vert a^5\Vert_{L^{\frac{6}{5}}_x} + \Vert a\Vert_{L^6_x}^5 \lesssim \Vert a \Vert_{L^6_x}^5.
\end{align*}

Direct computation yields
\begin{align*}
  -v \cdot \nabla_x \psi_a  &  = -5 \Big(a^5- \frac{1}{|\O|}\int a^5 \dd x \Big) \chi_0 - \sum_{i,j=1}^3 \p^2_{ij} \phi_a (\mathbf{I}-\mathbf{P})(v_iv_j(|v|^2-10)\mu^{1/2})   .
\end{align*}

Then LHS of~\eqref{steady_weak_formulation} becomes
\begin{align*}
  LHS  &  =  \frac{5}{\e^s} \int_\O \Big[ a^6 - \Big(\frac{1}{|\O|}\int_\O a^5 \dd x \Big) a(x) \Big] \dd x   + \frac{1}{\e^s}E_a = \frac{5}{\e^s} \Vert a\Vert_{L^6_x}^6 + \frac{1}{\e^s}E_a,
\end{align*}
where we have applied the conservation in mass~\eqref{mass_conserve}. Here $E_a$ satisfies
\begin{align*}
  |E_a|  &   \lesssim o(1) \Vert \nabla^2 \phi_a \Vert_{L^{\frac{6}{5}}_x}^{\frac{6}{5}} + \Vert (\mathbf{I}-\mathbf{P})f\Vert_{L^6_{x,v}}^6 \lesssim o(1)\Vert a\Vert_{L^6_x}^6 + \Vert (\mathbf{I}-\mathbf{P})f\Vert_{L^6_{x,v}}^6.
\end{align*}

The contribution of $K_1$ vanishes due to zero Neurmann boundary condition. $K_2$ and $K_3$ are estimated by the same computation as~\eqref{K2_b_l6_bdd} and~\eqref{K3_b_l6_bdd} with replacing $b$ by $a$. Then we conclude~\eqref{a_l6}.

\textit{Proof of~\eqref{c_l6}}.
We choose $\phi_c$ to be the solution of the following problem:
\begin{align*}
  -\Delta \phi_c   = c^5 - \frac{1}{|\O|}\int c^5 \dd x , \ \
   \frac{\p \phi_c}{\p n} = 0 \text{ on } \p \O, \ \ \int \phi_c \dd x = 0.
   \end{align*}
We take the following test function 
\begin{align}
  \psi  & :=\psi_c = \sum_{i=1}^3 \p_i \phi_c v_i(|v|^2-5) \mu^{1/2} =\sum_{i=1}^3 \sqrt{10}  \p_i \phi_c \hat{B}_i. \label{test_c}
\end{align}

A direct computation leads to
\begin{align*}
  -v\cdot \nabla_x \psi_c  &   = -\frac{5\sqrt{6}}{3} \Big(c^5- \frac{1}{|\O|}\int c^5 \dd x \Big) \chi_4 - \sum_{i,j=1}^3 \p_{ij}^2\phi_c (\mathbf{I}-\P)(v_iv_j (|v|^2-5)\mu^{1/2}) .
\end{align*}
Then LHS of~\eqref{steady_weak_formulation} reads
\begin{align*}
  LHS  & = \frac{5\sqrt{6}}{3\e^s} \int_{\O} \Big[ c^6 - \Big(\frac{1}{|\O|}\int c^5 \dd x \Big)c(x) \Big] \dd x + \frac{1}{\e^s} E_c = \frac{5\sqrt{6}}{3\e^s}  \Vert c\Vert_{L^6_x}^6 + \frac{1}{\e^s} E_c,
\end{align*}
where we have applied the conservation in energy~\eqref{energy_conserve}, and $E_c$ satisfies
\begin{align*}
  |E_c|  &  \lesssim o(1)\Vert \nabla^2 \phi_c \Vert_{L^{6/5}_x}^{6/5}   + \Vert (\mathbf{I}-\mathbf{P})f\Vert_{L^6_{x,v}}^6 \lesssim o(1)\Vert c\Vert_{L^6_x}^6 + \Vert (\mathbf{I}-\mathbf{P})f\Vert_{L^6_{x,v}}^6 .
\end{align*}

Again the contribution of $K_1$ vanishes, and $K_2$ and $K_3$ are bounded as~\eqref{K2_b_l6_bdd} and~\eqref{K3_b_l6_bdd} with replacing by $c$. Then we conclude~\eqref{c_l6}.

\end{proof}

\begin{proof}[\textbf{Proof of Theorem \ref{thm:l6}}]
Theorem \ref{thm:l6} follows by combining Lemma \ref{lemma:l6_b} and Lemma \ref{lemma:l6_a_c}.
\end{proof}

\section{$L^2$ estimate and proof of Theorem \ref{thm:coercive}}\label{sec:l2}

In this section, we will conclude the $L^2$ estimate (Theorem \ref{thm:coercive}). To be specific, in Section \ref{sec:l2_sketch}, we will prove the estimate of $b$ and estimate of $a,c$ in Lemma \ref{lemma:b} and Lemma \ref{lemma:a_c}. We sketch the proof since it shares some similarities to the $L^6$ estimate following the test function method. In Section \ref{sec:4.1}, we conclude Theorem \ref{thm:coercive}, and Corollary \ref{coro:boltzmann}, \ref{coro:landau}.

\subsection{$L^2$ estimate of $b$ and $a,c$}\label{sec:l2_sketch}

{\color{black}In Lemma \ref{lemma:b} and Lemma \ref{lemma:a_c}, we provide $L^2$ estimate of the macroscopic quantities following the test function method used in Section \ref{sec_l6}. The main difference in the proof is an additional estimate to the temporal derivative of the test function. The other estimates are the same as Lemma \ref{lemma:l6_b} and Lemma \ref{lemma:l6_a_c}.

\begin{lemma}\label{lemma:b}
Suppose all assumptions in Theorem \ref{thm:coercive} hold. For any $\delta_2>0$ and some constant $C_2>0$, we have the following estimate for the momentum:
\begin{align}
& \frac{1}{\e^s}(1-C_2\delta_2) \int_s^t \Vert b\Vert_{L^2_x}^2\dd \tau  \notag\\
 & \leq G_b(t) - G_b(s) + \frac{C_2 \delta_2}{\e^s} \int_s^t \Vert a\Vert_{L^2_x}^2\dd \tau +  \frac{C_2}{\e^s \delta_2} \int_s^t \Vert c\Vert_{L^2_x}^2\dd \tau \notag \\
 & + \frac{C_2}{\e^s \delta_2} \int_s^t \Vert (\mathbf{I}-\mathbf{P})f\Vert_{L^2_{x,v}}^2  \dd \tau+ \frac{C_2}{\delta_2 \e^{s+2k}} \int_s^t \Vert \mu^{1/4}\mathcal{L}_\alpha f \Vert_{L^2_{x,v}}^2 \dd \tau + \frac{C_2\e^s}{\delta_2}\int_s^t \Vert g\Vert_{L^2_{x,v}}^2 \dd \tau.  \label{b_estimate}
\end{align}
{\color{black} Here $|G_b(t)|\lesssim \Vert f(t)\Vert_{L^2_{x,v}}^2$ is defined in \eqref{Gb}.}

\end{lemma}

}

\begin{proof}

For the proof we rewrite the weak formulation \eqref{steady_f}: for all $t,s \in \R$
\begin{equation}\label{weak_formula} 
\begin{split}
 & \int_{\O\times \mathbb{R}^3} \{\psi f(t)-\psi f(s)\}\dd x \dd v - \int_s^{t}\int_{\O\times \mathbb{R}^3} f \p_\tau \psi \dd x \dd v \dd \tau \\
    & = \frac{1}{\e^s}\int_s^{t}\int_{\O\times \mathbb{R}^3} v \cdot \nabla_x \psi f\dd x \dd v \dd \tau - \frac{1}{\e^s}\int_s^{t} \int_{\gamma} \psi f \dd \gamma \dd \tau \\
    & + \frac{1}{\e^{k+s}}\int_s^t \int_{\O\times \mathbb{R}^3} (\mathcal{L}_\alpha f)\psi \dd v \dd x \dd \tau + \int_s^t \int_{\O\times \mathbb{R}^3} g\psi \dd x \dd v \dd \tau \\
    & := I_1+I_2+I_3+I_4.  
\end{split}
\end{equation}
\hide
{\color{black}This weak formulation holds for both linear operator $\mathcal{L}_\alpha$: Boltzmann ($\alpha= B$ and \eqref{Q_B}) or Landau ($\alpha=L$ and \eqref{Q_L}). The proof is identical for both cases except only $I_3$.}
\unhide

We consider the weak formulation \eqref{weak_formula} with special test functions. First we choose test function as
\begin{equation*}
    \psi_j = \phi({\color{black}t},x)\chi_j \ \  {\color{black}\text{for} \ \  j=1,2,3}. 
\end{equation*}
Then immediately we can conclude that $I_4=0$ {\color{black}from \eqref{Pg=0}}, and $I_3=0$ from~\eqref{collision_inv}. 

{\color{black}Taking $s=t-\triangle$ and $ t=t $, the} LHS of~\eqref{weak_formula} becomes
\begin{align*}
   \text{LHS of~\eqref{weak_formula}}& = \int_{\O} \{ {\color{black} b_j(t ) - b_j(t-\triangle)}\} \phi \dd x. 
\end{align*}
\hide
{\color{black}[[ ALERT!!: Since we will choose $\phi$ depends on the time variable, the second term of LHS including $\p_\tau \psi$ which is not vanishing. You can fix this by choosing $\psi$ depends on $t$ and choose $s=t-\triangle$ and $t=t$. Then $\psi$ does not depend on $\tau$ and by passing a limit $\triangle \rightarrow 0$ we will get what you want.  ]]}
\unhide

Then {\color{black}for each $j=1,2,3,$}
\Be\begin{split}\label{v_nabla_psi}
    & \frac{1}{\e^s}  v\cdot \nabla_x \psi_j =  \frac{1}{\e^s}\sum_{i=1}^3 \p_i \phi(x) v_i v_j \mu^{1/2} \\
    & =  \frac{1}{\e^s} \Big[ \sum_{i=1}^3 \p_i \phi(x) \mu^{1/2}\left(v_iv_j - \frac{\delta_{ij}}{3}|v|^2\right) + \sum_{i=1}^3 \p_i \phi(x) \mu^{1/2} \frac{\delta_{ij}}{3}|v|^2  \Big] \\
    & =  \frac{1}{\e^s} \Big[ \sum_{i=1}^3 \p_i\phi   \hat{A}_{ij}     + \p_j \phi \mu^{1/2} \left(\frac{|v|^2}{3}-1\right) + \p_j \phi \mu^{1/2}  \Big]  \\
    & =  \frac{1}{\e^s} \Big[ \p_j \phi  \left(\chi_0 + \frac{\sqrt{6}}{3}\chi_4\right)+ \sum_{i=1}^3 \p_i \phi \hat{A}_{ij} \Big],
\end{split}
\Ee
{\color{black} where we have used $\hat A_{ij}$ in \eqref{B_ij} and \eqref{fact}.}

For RHS of~\eqref{weak_formula}, from~\eqref{B_ij_property} and \eqref{v_nabla_psi}, we have
\begin{align*}
   I_1 &  =  \frac{1}{\e^s}\int_{t-\triangle}^{t} \int_{\O}  \{\p_j \phi (a+\frac{\sqrt{6}}{3} c ) \} \dd x \dd \tau \\
   & +  \frac{1}{\e^s}\int_{t-\triangle}^{t} \int_{\O} \sum_{i=1}^3  \p_i \phi  \langle \hat{A}_{ij} , (\mathbf{I}-\P)f \rangle \dd x \dd \tau.
\end{align*}
Taking ${\color{black}\triangle} \to 0 $ yields
\begin{align}
  \int_{\O} \phi (t,x) \p_\tau b_j(t,x)   \dd x  &  =   \frac{1}{\e^s}\int_{\O}  \{\p_j \phi (t,x)  (a (t,x)  +\frac{\sqrt{6}}{3} c (t,x)   ) \} \dd x \notag\\
  & + \frac{1}{\e^s} \int_{\O} \sum_{i=1}^3  \p_i \phi  \langle \hat{A}_{ij} ,(\mathbf{I}-\P)f (t,x)  \rangle \dd x - \frac{1}{\e^s}\int_\gamma  f (t,x)  \phi (t,x)  \chi_j   \dd \gamma. \label{delta_to_0}
\end{align}

Let $\Phi_b = (\Phi_b^1, \Phi_b^2, \Phi_b^3)$ solve
\begin{align}
  -\Div(\nablaS \Phi_b)  &    = \p_\tau  b(t) \text{ in }\O \notag
  \\
  \Phi_b \cdot n& = 0 \label{system_p_b} \text{ on }\p\O\\
  \nablaS \Phi_b  n  & = (\nablaS \Phi_b : n \otimes n)n \text{ on }\p\O. \notag 
\end{align}

By~\eqref{angular_momentum_preserve} and Theorem \ref{thm:symmetric_poisson}, there exists a unique solution $\Phi_b \in \mathcal{X}_0$ defined in~\eqref{space} such that $\Vert \Phi_b \Vert_{H^2_x} \lesssim \Vert \p_\tau b\Vert_{L^2_x}$ and $\Phi_b$ satisfies~\eqref{system_p_b} a.e. Choose $\phi = \Phi_b^j$, taking summation in $j=1,2,3$ leads to $I_2 = 0$ by $\Phi_b \cdot n = 0$: 
\begin{align*}
    & \sum_{j=1}^3 \int_{\gamma} f \Phi_b^j \chi_j \dd \gamma = 2 \int_{\p\O} \int_{n\cdot v>0}  (\Phi_b\cdot n) (n\cdot v)^2 \mu^{1/2} f   \dd v \dd S_x =0.
\end{align*}

Combining~\eqref{delta_to_0} with $I_2=0$, we have
\begin{align*}
   \sum_{j=1}^3 \int_{\O}|\nablaS \Phi_b^j|^2 & = -\sum_{j=1}^3 \int_{\O} \Div(\nablaS \Phi_b^j)   \Phi_b^j = \sum_{j=1}^3 \int_{\O}\Phi_b^j \p_\tau b_j \dd x \\
   & \lesssim  \frac{1}{\e^s} \Vert \nabla \Phi_b\Vert_{L^2_x} (\Vert a\Vert_{L^2_x} +\Vert c\Vert_{L^2_x})  + \frac{1}{\e^s} \Vert \nabla \Phi_b\Vert_{L^2_x} \Vert (\mathbf{I}-\P)f\Vert_{L^2_{x,v}}.
\end{align*}
In the first equality we applied the same computation as~\eqref{bdr_vanish}.

Since $\Phi_b \in \mathcal{X}_0$ in~\eqref{space}, applying the Korn's inequality in Theorem \ref{lemma:estimate_nabla} \\
and $P_\O\big(\int_{\O}\nabla^a \Phi_b \dd x\big) = 0$ we have
\begin{align*}
    \Vert \Phi_b\Vert^2_{H^1_x} \lesssim \Vert \nablaS \Phi_b \Vert^2_{L^2_x} = \sum_{j=1}^3 \int_{\O} |\nablaS \Phi_b^j|^2,
\end{align*}
and this leads to
\begin{align}
 \Vert \Phi_b \Vert_{H^1_x}   &   \lesssim   \frac{1}{\e^s}\Big[\Vert a\Vert_{L^2_x}+ \Vert c\Vert_{L^2_x}+ \Vert (\mathbf{I}-\P)f \Vert_{L^2_{x,v}} \Big]. \label{est_Phi_b}
\end{align}

Then we start the estimate of $b$. We rearrange~\eqref{weak_formula} into the following form:
\begin{equation}\label{weak_formula_2}
    \begin{split}
        & - \frac{1}{\e^s}\int_s^t \int_{\O\times \mathbb{R}^3} v\cdot \nabla_x \psi f \dd x \dd v \dd \tau   \\
    & = \int_{\O\times \mathbb{R}^3}\{-\psi f(t) + \psi f(s)\} \dd x \dd v +  \int_s^t \int_{\O\times \mathbb{R}^3} f \p_\tau \psi \dd x \dd v  \dd \tau  - \int_s^t \int_{\gamma} \psi f  \dd \gamma  \dd \tau   \\
    & +   \frac{1}{\e^{s+k}}\int_s^t  \int_{\O\times \mathbb{R}^3} \mathcal{L}_\alpha f \psi  \dd x \dd v \dd \tau + \int_s^t \int_{\O\times \mathbb{R}^3} g \psi \dd  x \dd v \dd \tau  \\
    & := \{G_\psi(t) - G_\psi(s)\} + J_1 + J_2 + J_3 + J_4 .   
    \end{split}
\end{equation}

We consider a variant of \eqref{system_b_5}:
\begin{align}
  -\Div(\nablaS \phi_b)  & = \frac{1}{2}b(\tau) \text{ in }\O \notag \\
  \phi_b \cdot n & = 0  \text{ on } \p\O\label{system_b}\\
  \nablaS \phi_b n& = (\nablaS \phi_b : n\otimes n)n \text{ on } \p\O .\notag
\end{align}
By~\eqref{angular_momentum} and Theorem \ref{thm:symmetric_poisson}, there is a unique solution $\phi_b$ satisfies the system~\eqref{system_b} with
\begin{equation}\label{phi_b_H2}
\Vert \phi_b\Vert_{H^2_x}\lesssim \Vert b\Vert_{L^2_x}.
\end{equation}

{\color{black} Then $G_b$ in \eqref{b_estimate} is defined as
\begin{align}
   G_b(t) &:=  G_{\psi_b}(t) = -\int_{\O \times \mathbb{R}^3} \psi_b f(t) \dd x \dd v. \label{Gb}
\end{align}
We apply H\"older inequality and \eqref{phi_b_H2} to have
\begin{align}
  |G_b(t)|  &  \leq \Vert f(t)\Vert_{L^2_{x,v}} \Vert \psi_b(t)\Vert_{L^2_x} \lesssim \Vert f(t)\Vert_{L^2_{x,v}} \Vert \phi_b(t)\Vert_{H^1_x} \lesssim \Vert f(t)\Vert_{L^2_{x,v}} \Vert b(t)\Vert_{L^2_x} \leq \Vert f(t)\Vert_{L^2_{x,v}}^2.  \label{Gb_bdd}
\end{align}
}

We choose test function as $\psi_b$ given in~\eqref{test_b_1}, with $\phi_b$ given by \eqref{system_b}. Following the same computation in~\eqref{test_b_derivative} and~\eqref{transport_b}, we have
\begin{align*}
  -v\cdot \nabla_x \psi_b  &  =   -\sum_{i,j,k=1}^3 \p_{kj}\phi_b^i (\mathbf{I}-\mathbf{P})(v_iv_jv_k \mu^{1/2}) - \sum_{i=1}^3 \chi_i [\Delta \phi_b^i + \p_i \Div(\phi_b)] \\
  & = -\sum_{i,j,k=1}^3 \p_{kj}\phi_b^i (\mathbf{I}-\mathbf{P})(v_iv_jv_k \mu^{1/2}) + \sum_{i=1}^3 \chi_i b_i.
\end{align*}

Thus LHS of~\eqref{weak_formula_2} becomes
\begin{align}
   LHS & =   \frac{1}{\e^s}\Big\{ \int_s^t  \int_{\O} |b|^2 \dd x \dd \tau -   \underbrace{\sum_{i,j,k=1}^3 \int_s^t  \int_\O \p_{kj}\phi_b^i \langle v_iv_jv_k \mu^{1/2}, (\mathbf{I}-\mathbf{P})f \rangle  \dd x \dd \tau}_{E_4}  \Big\}, \label{estimate_LHS}
   \end{align}
where, by~\eqref{phi_b_H2}, for some $\delta_2\ll 1$, 
\begin{align*}
  |E_4|  & \lesssim    \delta_2 \int_s^t \Vert b\Vert_{L^2_x}^2\dd \tau  + \frac{1}{\delta_2}  \int_s^t \Vert \mu^{1/4}(\mathbf{I}-\mathbf{P})f\Vert_{L^2_{x,v}}^2 \dd \tau. 
\end{align*}

Next we estimate $J_i, 1\leq i\leq 4$ in~\eqref{weak_formula_2}. Note that $\Phi_b = \p_\tau \phi_b$, where we have an estimate of $\Phi_b$ in~\eqref{est_Phi_b}. Applying the representation of $\psi_b$ in~\eqref{test_b_2}, from~\eqref{B_ij_property}, we have 
\begin{align}
 |J_1|   &  \lesssim  \int_s^t  \int_{\O} \big( |c(x)| + |\langle \hat{A}_{ij}, (\mathbf{I}-\mathbf{P})f\rangle|\big) |\nabla_x \Phi_b| \dd x \dd \tau \notag\\
    &\lesssim \e^s \delta_2 \int_s^t \Vert \nabla \Phi_b\Vert^2_{L^2_x} \dd \tau + \frac{1}{\e^s\delta_2} \int_s^t [\Vert c\Vert_{L^2_x}^2+\Vert \mu^{1/4}(\mathbf{I}-\mathbf{P})f\Vert_{L^2_{x,v}}^2] \dd \tau \notag\\
    &\lesssim \frac{\delta_2}{\e^s} \int_s^t \Vert a\Vert_{L^2_x}^2\dd \tau + \frac{1}{\e^s \delta_2} \int_s^t \Vert c\Vert_{L^2_x}^2\dd \tau + \frac{1}{\e^s\delta_2} \int_s^t \Vert \mu^{1/4}(\mathbf{I}-\mathbf{P})f \Vert_{L^2_{x,v}}^2 \dd \tau. \label{J_1_estimate}
\end{align}
In the last line we used~\eqref{est_Phi_b}.

For $J_2$, since~\eqref{system_b} shares the same boundary condition as~\eqref{system_b_5}, the estimate of $J_2$(the boundary contribution), follows from~\eqref{K_1_estimate}: 
\begin{align}
    J_2 = 0. \label{J_2_estimate}
\end{align}

For $J_3$, we apply Young's inequality to have
\begin{align}
  |J_3|  & \lesssim  \frac{1}{\e^{s+k}}\int_s^t \Big[ \frac{1}{\e^k\delta_2} \Vert \mu^{1/4}\mathcal{L}_\alpha f \Vert_{L^2_{x,v}}^2 + \e^k\delta_2 \Vert \nabla \phi_b\Vert^2_{L^2_x} \Big] \dd \tau   \notag\\ 
  & \lesssim \frac{\delta_2}{\e^{s}} \int_s^t \Vert b\Vert_{L^2_x}^2\dd \tau +\frac{1}{\delta_2 \e^{s+2k}} \int_s^t \Vert \mu^{1/4}\mathcal{L}_\alpha f \Vert_{L^2_{x,v}}^2 \dd \tau. \label{J_3_estimate}
\end{align}
Here we used~\eqref{phi_b_H2}.

For $J_4$, similar to the estimate of $J_3$, we have
\begin{align}
  |J_4|  &  \lesssim \frac{\delta_2}{\e^s} \int_s^t \Vert b\Vert_{L^2_x}^2\dd \tau +\frac{\e^s}{\delta_2} \int_s^t \Vert g\Vert_{L^2_{x,v}}^2 \dd \tau. \label{J_4_estimate}
\end{align}

Collecting~\eqref{Gb_bdd},~\eqref{estimate_LHS}, \eqref{J_1_estimate}, \eqref{J_2_estimate}, \eqref{J_3_estimate} and~\eqref{J_4_estimate}, for some constant $C_2$, we conclude~\eqref{b_estimate}.

\end{proof}

\begin{lemma}\label{lemma:a_c}
For any $\delta_1>0$ and some constant $C_1>0$, we have the following estimate for the energy:
\begin{align}
& \frac{1}{\e^s}(1-C_1\delta_1) \int_s^t \Vert c\Vert_{L^2_x}^2\dd \tau  \notag\\
 & \leq G_c(t) - G_c(s) + \frac{C_1 \delta_1}{\e^s} \int_s^t \Vert b\Vert_{L^2_x}^2\dd \tau + \frac{C_1}{\e^s\delta_1}\int_s^t \Vert \mu^{1/4}(\mathbf{I}-\mathbf{P})f\Vert_{L^2_{x,v}}^2 \dd \tau  \notag \\
 & + \frac{C_1}{\e^{s+2k}\delta_1} \int_s^t \Vert \mu^{1/4} \mathcal{L}_\alpha f \Vert_{L^2_{x,v}}^2 \dd \tau + \frac{C_1\e^s}{\delta_1}\int_s^t \Vert g\Vert_{L^2_{x,v}}^2 \dd \tau.  \label{c_estimate}
\end{align}

For any $\delta_3>0$ and some constant $C_3>0$, we have the following estimate for the mass:
\begin{align}
& \frac{1}{\e^s}(1-C_3\delta_3) \int_s^t \Vert a\Vert_{L^2_x}^2\dd \tau  \notag\\
 & \leq G_a(t) - G_a(s)  +  \frac{C_3}{\e^s } \int_s^t \Vert b\Vert_{L^2_x}^2\dd \tau + \frac{C_3}{\e^s \delta_3} \int_s^t \Vert \mu^{1/4}(\mathbf{I}-\mathbf{P})f\Vert_{L^2_{x,v}}^2 \dd \tau\notag \\
 & + \frac{C_3}{\e^{s+2k}\delta_3} \int_s^t \Vert \mu^{1/4}\mathcal{L}_\alpha f \Vert_{L^2_{x,v}}^2 \dd \tau + \frac{C_3\e^s}{\delta_3}\int_s^t \Vert g\Vert_{L^2_{x,v}}^2 \dd \tau.  \label{a_estimate}
\end{align}

{\color{black} Here $|G_a(t)|,|G_c(t)|\lesssim \Vert f(t)\Vert_{L^2_{x,v}}^2$ are defined in \eqref{Ga} and \eqref{Gc} respectively.}

\end{lemma}

\begin{proof}
We use the weak formulation~\eqref{weak_formula} and~\eqref{weak_formula_2}.

\textit{Proof of~\eqref{c_estimate}}.
We choose $\psi = \phi(x)\chi_4$ in~\eqref{weak_formula}. Then the transport operator on $\psi$ becomes
\begin{equation}\label{transport_c_1}
v\cdot \nabla_x \psi =  \sum_{i=1}^3 \frac{\sqrt{6}}{3} \p_i \phi \chi_i + \sum_{i=1}^3  \frac{\sqrt{15}}{3} \p_i \phi \hat{B}_i.
\end{equation}
Take $[s,t] = [t-\triangle,t]$, LHS of~\eqref{weak_formula} becomes
\[LHS = \int_{\O} \{ c(t) - c(t-\triangle)\} \phi \dd x.\]
By~\eqref{transport_c_1} we have
\[I_1 = \frac{1}{\e^s}\int^t_{t-\triangle} \int_\O \Big\{\frac{\sqrt{6}}{3}(b\cdot \nabla_x \phi) + \frac{1}{\e^s}\sum_{i=1}^3  \frac{\sqrt{15}}{3} \p_i \phi \langle \hat{B}_i, (\mathbf{I}-\mathbf{P})f \rangle \Big\} \dd x \dd \tau,\]
while $I_2=I_3=I_4 = 0$. Choosing $\phi=1$ leads to
\[\int_\O \p_\tau c \dd x =0.\]

Then for fixed $t$ we choose $\phi = \Phi_c$ such that
\begin{align*}
  -\Delta \Phi_c  & = \p_\tau c(t) \text{ in } \O , \ 
  \frac{\p \Phi_c }{\p n}  = 0  \text{ on } \p\O , \ 
    \int_{\O} \Phi_c \dd x = 0.
\end{align*}
This leads to
\[\int_\O |\nabla_x \Phi_c|^2 \dd x = \int_{\O}\Phi \p_\tau c \dd x  \lesssim \frac{1}{\e^s}\Vert \nabla \Phi_c\Vert_{L^2_x} (\Vert b\Vert_{L^2_x}+ \Vert \mu^{1/4}(\mathbf{I}-\P)f\Vert_{L^2_{x,v}}).\]
Thus from the standard Poincaré inequality we conclude
\begin{equation}\label{estimate_c_1}
\Vert \Phi_c\Vert_{H^1_x} \lesssim  \frac{1}{\e^s}(\Vert b\Vert_{L^2_x}+ \Vert \mu^{1/4}(\mathbf{I}-\P)f\Vert_{L^2_{x,v}} ).
\end{equation}

Then we use the formulation in~\eqref{weak_formula_2}. We let $\phi_c$ be a solution of the following problem
\begin{align}
  -\Delta \phi_c   = c(t) \text{ in } \O  , \ 
  \frac{\p \phi_c }{\p n}   = 0  \text{ on } \p\O, \ 
    \int_{\O} \phi_c \dd x = 0. \label{l2_phi_c}
\end{align}
Then we use the same test function $\psi_c$ as~\eqref{test_c} with $\phi_c$ given by \eqref{l2_phi_c}. 

{\color{black} Then $G_c(t)$ in \eqref{c_estimate} is defined as
\begin{align}
   G_c(t):= & G_{\psi_c}(t) = -\int_{\O\times \mathbb{R}^3} \psi_c(t) f(t) \dd x \dd v. \label{Gc}
\end{align}
Applying H\"older inequality and elliptic theory,
\begin{align*}
    |G_c(t)|& \lesssim \Vert f(t)\Vert_{L^2_{x,v}} \Vert \phi_c(t)\Vert_{H^1_x} \lesssim \Vert f(t)\Vert_{L^2_{x,v}}\Vert c(t)\Vert_{L^2_x} \lesssim \Vert f(t)\Vert_{L^2_{x,v}}^2.
\end{align*}

}

A direct computation leads to
\begin{align}
   -v\cdot \nabla_x \psi_c &  = -\frac{5\sqrt{6}}{3}\Delta \phi_c \chi_4  - \sum_{i,j=1}^3 \p_{ij}^2\phi_c (\mathbf{I}-\P)(v_iv_j (|v|^2-5)\mu^{1/2}).\notag
\end{align}
Thus the LHS of~\eqref{weak_formula_2} is
\begin{align*}
  LHS  &  = \frac{5\sqrt{6}}{3\e^s} \int_s^t \int_{\O} c^2 \dd x\dd \tau - \frac{1}{\e^s}\sum_{i,j=1}^3 \int_s^t  \int_{\O} \p_{ij}^2 \phi_c  \langle (\mathbf{I}-\mathbf{P})f, v_iv_j (|v|^2-5)\mu^{1/2}\rangle \dd x \dd \tau \\
  & = \frac{5\sqrt{6}}{3\e^s}\int_s^t \int_\O c^2 \dd x \dd \tau  + E_3,
\end{align*}
where, for any $\delta_1>0$,
\[|E_3|\lesssim \frac{\delta_1}{\e^s}\int_s^t \Vert c\Vert_{L^2_x}^2\dd \tau + \frac{1}{\e^s \delta_1}\int_s^t \Vert \mu^{1/4}(\mathbf{I}-\P)f\Vert_{L^2_{x,v}}^2 \dd \tau.\]
For $J_1$, using the fact that $\Phi_c = \p_\tau \phi_c $ and the estimate for $\Phi_c $ in~\eqref{estimate_c_1} we have
\begin{align*}
  |J_1|  &  \lesssim  \delta_1\e^s\int_s^t \Vert \Phi_c \Vert^2_{H^1_x} \dd \tau  + \frac{1}{\delta_1\e^s} \int_s^t \Vert \mu^{1/4}(\mathbf{I}-\P)f\Vert_{L^2_{x,v}}^2  \dd \tau\\
  & \lesssim \frac{\delta_1}{\e^s} \int_s^t \Vert b\Vert_{L^2_x}^2\dd \tau + \frac{1}{\delta_1\e^s} \int_s^t \Vert \mu^{1/4}(\mathbf{I}-\P)f\Vert_{L^2_{x,v}}^2  \dd \tau.
\end{align*}
Applying the boundary condition of $\phi_c$ and $f$ we will have $J_2=0$. The estimate of $J_3$ and $J_4$ are the same of~\eqref{J_3_estimate} and~\eqref{J_4_estimate} with replacing $b$ by $c$. Then we conclude~\eqref{c_estimate}.

\textit{Proof of~\eqref{a_estimate}}. We choose $\psi = \phi(x)\chi_0$. Direct computation yields
\begin{align*}
 v\cdot \nabla_x \psi   &   =\sum_{i=1}^3 \p_i \phi \chi_i .
\end{align*}
Then in~\eqref{weak_formula},
\[I_1 = \frac{1}{\e^s}\int_{t-\triangle}^t \int_{\O} (b\cdot \nabla_x \phi) \dd x \dd \tau,\]
while $I_2=I_3=I_4 = 0$.

We take $[s,t] = [t-\triangle,t]$ in~\eqref{weak_formula}, then LHS becomes 
\[\int_{\O} \{ a(t) - a(t-\triangle)\} \phi \dd x.\]
Taking $\triangle \to 0$ yields
\[\int_{\O} \phi \p_\tau a \dd x = \frac{1}{\e^s}\int_\O (b\cdot \nabla_x \phi) \dd x.\]
For $\phi\in H^1_x$, we have
\begin{align*}
  \Big|\int_\O \phi \p_\tau a  \dd x \Big|  &  \lesssim \frac{1}{\e^s}\Vert b\Vert_{L^2_x}\Vert \phi\Vert_{H^1_x} ,
\end{align*}
which leads to
\begin{align*}
  \Vert \p_\tau a \Vert_{H^{-1}_0(\O)}  &   \lesssim \frac{1}{\e^s}\Vert b\Vert_{L^2_x}.
\end{align*}

Setting $\phi=1$ leads to $\int_\O \p_\tau a \dd x = 0$, which is the solvability condition for the system
\begin{align*}
  -\Delta \Phi_a   = \p_\tau a(t) \text{ in } \O  , \ 
  \frac{\p \Phi_a }{\p n}   = 0  \text{ on } \p\O ,\ 
    \int_{\O} \Phi_a \dd x  = 0.
\end{align*}
For such test function we have
\begin{equation}\label{estimate_a_1}
\Vert \Phi_a\Vert_{H^1_x}\lesssim \Vert \p_\tau a\Vert_{H_0^{-1}(\O)} \lesssim \frac{1}{\e^s}\Vert b\Vert_{L^2_x}.
\end{equation}

Next we choose $\phi_a$ to be the solution of the following problem:
\begin{align}
  -\Delta \phi_a   = a(t) \text{ in } \O  , \  
  \frac{\p \phi_a }{\p n}   = 0  \text{ on } \p\O , \ 
    \int_{\O} \phi_a \dd x  = 0. \label{l2_phi_a}
\end{align}
Then we use the same test function $\psi_a$ as~\eqref{test_a}, with $\phi_a$ given by \eqref{l2_phi_a}. 

{\color{black} Then $G_a(t)$ in \eqref{a_estimate} is defined as
\begin{align}
    & G_a(t):= G_{\psi_a}(t) = -\int_{\O\times \mathbb{R}^3} \psi_a(t)f(t) \dd x \dd v. \label{Ga}
\end{align}
By H\"older inequality and elliptic theory, we have
\begin{align*}
    &|G_a(t)|\lesssim \Vert f(t) \Vert_{L^2_{x,v}} \Vert \phi_a(t)\Vert_{H^1_x} \lesssim \Vert f(t)\Vert_{L^2_{x,v}} \Vert a(t)\Vert_{L^2_x} \leq \Vert f(t)\Vert_{L^2_{x,v}}^2.
\end{align*}
}

Direct computation yields
\begin{align}
  -v\cdot \nabla_x \psi  & =-5 \Delta \phi_a  \chi_0 - \sum_{i,j=1}^3 \p^2_{ij} \phi_a (\mathbf{I}-\mathbf{P})(v_iv_j(|v|^2-10)\mu^{1/2})  . \notag
\end{align}
Thus LHS of~\eqref{weak_formula_2} becomes 
\begin{align*}
  LHS  &  = \frac{5}{\e^s} \int_s^t \Vert a\Vert_{L^2_x}^2\dd \tau - \frac{1}{\e^s}\sum_{i,j=1}^3 \int_s^t \int_{\O} \p_{ij}^2 \phi_a \langle v_iv_j (|v|^2-10)\mu^{1/2},(\mathbf{I}-\mathbf{P})f\rangle \dd x \dd \tau\\
  & = \frac{5}{\e^s} \int_s^t \Vert a\Vert_{L^2_x}^2\dd \tau + E_5,
\end{align*}
where, for any $\e_3>0$,
\begin{align*}
  |E_5|  & \lesssim \frac{\e_3^2}{\e^s} \int_s^t \Vert a\Vert_{L^2_x}^2\dd \tau  + \frac{1}{\e_3^2 \e^s} \int_s^t \Vert \mu^{1/4}(\mathbf{I}-\P)f\Vert_{L^2_{x,v}}^2 \dd \tau.
\end{align*}
For $J_1$ in~\eqref{weak_formula_2}, from $\Phi_a = \p_\tau \phi_a$ and~\eqref{estimate_a_1}, we have
\begin{align*}
  |J_1|  &  \lesssim  \e^s\int_s^t \Vert \Phi_a \Vert^2_{H^1_x} \dd \tau + \frac{1}{\e^s} \int_s^t \Vert b\Vert_{L^2_x}^2 \dd \tau + \frac{1}{\e^s} \int_s^t \Vert \mu^{1/4}(\mathbf{I}-\P)f\Vert_{L^2_{x,v}}^2  \dd \tau\\
  & \lesssim \frac{1}{\e^s} \int_s^t \Vert b\Vert_{L^2_x}^2\dd \tau + \frac{1}{\e^s} \int_s^t \Vert \mu^{1/4}(\mathbf{I}-\P)f\Vert_{L^2_{x,v}}^2  \dd \tau .
\end{align*}
The boundary condition of $\Phi_a$ and $f$ lead to $J_2=0$. The estimate of $J_3$ and $J_4$ are the same of~\eqref{J_3_estimate} and~\eqref{J_4_estimate} with replacing $b$ by $a$. Then we conclude~\eqref{a_estimate}.

\end{proof}

\subsection{Theorem \ref{thm:coercive}, Corollary \ref{coro:boltzmann} and Corollary \ref{coro:landau}.}\label{sec:4.1}

Theorem \ref{thm:coercive} follows by combining Lemma \ref{lemma:b} and Lemma \ref{lemma:a_c}.

\begin{proof}[\textbf{Proof of Theorem \ref{thm:coercive}}]
We choose 
\[\delta_1 = (2^14 C_1 C_2^2 C_3^2)^{-1}, \ \delta_2 = (2^8 C_2 C_3^2)^{-1}, \ \delta_3=(4C_3)^{-1}.\]
We set 
\[\delta_a = (2^8C_2 C_3^2)^{-1}, \ \delta_b = (2^{11/2}C_2C_3)^{-1}.\]
A direct computation of $\delta_a \times \eqref{a_estimate} + \delta_b \times\eqref{b_estimate} + \eqref{c_estimate}$ yields
\begin{align*}
    &\frac{\delta_a}{4\e^s}  \int_s^t (\Vert a\Vert_{L^2_x}^2+ \Vert b\Vert_{L^2_x}^2+ \Vert c\Vert^2_{L^2_x}) \dd \tau \\
    & \leq G(t)-G(s) + \frac{1}{\e^s } \int_s^t \Vert \mu^{1/4}(\mathbf{I}-\mathbf{P})f\Vert_{L^2_{x,v}}^2  \dd \tau+ \frac{1}{ \e^{s+2k}} \int_s^t \Vert \mu^{1/4}\mathcal{L}_\alpha f \Vert_{L^2_{x,v}}^2 \dd \tau + C\e^s \int_s^t \Vert g\Vert_{L^2_{x,v}}^2 \dd \tau,
\end{align*}
where $C=2^{14}C_1^2 C_2^2 C_3^2$ and
\begin{align}
    & |G(t)| = |G_a(t)+G_b(t)+G_c(t)| = \Big|\int_\O -(\psi_a + \psi_b +\psi_c)f(t) \dd x \Big| \lesssim \Vert f(t)\Vert_{L^2_{x,v}}^2. \label{G_def}
\end{align}
Then we conclude the theorem.

\end{proof}

\begin{proof}[\textbf{Proof of Corollary \ref{coro:boltzmann}}]
This corollary follows directly from the control of the linearized Boltzmann operator:
\begin{align*}
    & \Vert \mathcal{L}_B f \Vert_{L^2_{x,v}} = \Vert \mathcal{L}_B(\mathbf{I}-\mathbf{P})f\Vert_{L^2_{x,v}} \lesssim \Vert \nu (\mathbf{I}-\mathbf{P})f\Vert_{L^2_{x,v}},
\end{align*}
where $\nu(v) =\sqrt{1+|v|^2}$ and $\nu(v)\mu^{1/4}(v)\lesssim 1$.
    
\end{proof}

To prove Corollary \ref{coro:landau}, we first cite several estimates about the Landau operator and weighted norm $\Vert \cdot \Vert_\sigma$ from \cite{Guo_landau_box}.

First we define the norm with respect to only velocity variable as
\begin{align*}
    | f|_\sigma^2 &:=\int_{ \mathbb{R}^3} \Big[ \sigma^{ij} \p_{v_i} f\p_{v_j} f  +\sigma^{ij} v_i v_j f^2 \Big]  \dd v ,\\
     |f|_{L^2}^2 & := \int_{\mathbb{R}^3} |f|^2 \dd v.
\end{align*}

\begin{lemma}[Corollary 1 in \cite{Guo_landau_box}]\label{lemma:sigma_lower}

There exists $c>0$ such that
\begin{align*}
   | f| ^2_{\sigma} & \geq c | [1+|v|]^{-\frac{1}{2}} f |_{L^2}^2 + c | [1+|v|]^{-3/2}\p_i f|_{L^2}^2 .
\end{align*}

\end{lemma}

\begin{lemma}[Lemma 1 and Lemma 5 in \cite{Guo_landau_box}]

The linear operator can be expressed as $\mathcal{L}_L=-A-K$, where
\begin{equation}\label{A_K}
\begin{split}
    Af & = \p_i [\sigma^{ij} \p_j f] - \sigma^{ij} v_i v_j f + \p_i \sigma^i f,\\
    Kf & = -\mu^{-1/2} \p_i \Big\{\mu\big[ \phi * \big\{\mu^{1/2}[\p_j f + v_j f] \big\}\big] \Big\}.    
\end{split}
\end{equation}
Here $\sigma^i := \sum_j \sigma^{ij}$, where $\sigma^{ij}$ is defined in~\eqref{sigma}.

Moreover, for any $m>1$, there exists $C(m)>0$ such that
\begin{equation}\label{A_K_upper}
\begin{split}
 & |\langle w^{2\theta}\p_i \sigma^i f,g\rangle| + |\langle w^{2\theta}Kf,g\rangle|    \\
 & \leq \frac{C}{m}| w^\theta f|_\sigma | w^\theta g|_\sigma + C(m) \Big\{\int_{|v|\leq C(m)} w^{2\theta}|f|^2 \dd v \Big\}^{1/2} \Big\{ \int_{|v|\leq C(m)} w^{2\theta}|g|^2 \dd v \Big\}^{1/2}.    
\end{split}
\end{equation}
Here $w^\theta$ is a weight defined as $w^\theta(v) = [1+|v|]^{\theta}$.

Hence
\begin{align}
  |\langle \mathcal{L}_L f , \mu^{1/4} g\rangle|  & = |\langle \mathcal{L}_L (\mathbf{I}-\mathbf{P})f,\mu^{1/4} g \rangle | \lesssim o(1)| g|_\sigma^2 + | (\mathbf{I}-\mathbf{P})f|_\sigma^2.\label{L_upper_bound}
\end{align}

\end{lemma}

\begin{proof}
We only need to prove~\eqref{L_upper_bound}. The first equality is a consequence of the fact that $\mathbf{P}f$ is in the null space of $\mathcal{L}_L$. For the second inequality, due to the extra factor $\mu^{1/4}$, from~\eqref{A_K} and~\eqref{A_K_upper}, we replace $w^{2\theta}$ by $\mu^{1/4}$, so that the contribution of $\p_i \sigma^i$ and $K$ in~\eqref{L_upper_bound} are bounded as
\begin{align*}
    & | \mu^{1/8}(\mathbf{I}-\mathbf{P})f|_\sigma | \mu^{1/8}g|_{\sigma} + | \mu^{1/8}(\mathbf{I}-\mathbf{P})f|_{L^2} | \mu^{1/8}g|_{L^2} \\
    & \lesssim  o(1)| g|_\sigma^2 + | (\mathbf{I}-\mathbf{P})f|_\sigma^2.
\end{align*}
Here we have applied Lemma \ref{lemma:sigma_lower} to control the $L^2$ norm.

For the contribution of rest terms in $A$ of~\eqref{A_K}, we compute as
\begin{align*}
  |\langle \sigma^{ij}v_i v_j (\mathbf{I}-\mathbf{P})f, \mu^{1/4} g\rangle|  & \lesssim   |\mu^{1/10} (\mathbf{I}-\mathbf{P})f|_{L^2}^2 + o(1)|\mu^{1/10}g|_{L^2}^2 \\
  &\lesssim o(1)| g|_\sigma^2 + | (\mathbf{I}-\mathbf{P})f|_\sigma^2.
\end{align*}
Through integration by part, 
\begin{align*}
   \langle \p_i[\sigma^{ij}\p_j (\mathbf{I}-\mathbf{P})f], \mu^{1/4}g \rangle &  \lesssim | \mu^{1/10}\p_i [(\mathbf{I}-\mathbf{P})f]|_{L^2}^2 + o(1) |\mu^{1/10}\p_i g|_{L^2}^2 + o(1)|\mu^{1/10} g|_{L^2}^2 \\
   & \lesssim |[(\mathbf{I}-\mathbf{P})f]|_{\sigma}^2 + o(1)|g|_{\sigma}^2,
\end{align*}
where we have applied Lemma \ref{lemma:sigma_lower}.

\end{proof}

\begin{proof}[\textbf{Proof of Corollary \ref{coro:landau}}]

By the definition of $\mathbf{P}f$ in~\eqref{projection} and the weighted norm in~\eqref{sigma_norm}, we have
\begin{align*}
\Vert \mathbf{P}f \Vert_\sigma^2 & \lesssim \int_\O \big[ a^2(x)+\sum_i b_i^2(x)+c^2(x)\big] \dd x \int_{\mathbb{R}^3}\big[(1+|v|+|v|^2)\mu^{1/2} \big]\dd v \lesssim \Vert \mathbf{P}f\Vert_{L^2_{x,v}}^2.
\end{align*}

By Lemma \ref{lemma:sigma_lower}, we have
\begin{align*}
    & \Vert \mu^{1/4}(\mathbf{I}-\mathbf{P})f\Vert_{L^2_{x,v}} \lesssim \Vert (\mathbf{I}-\mathbf{P})f\Vert_\sigma.    
\end{align*}

For the contribution of the linearized operator $\mathcal{L}_L f$, we slightly modify the estimate in~\eqref{J_3_estimate}. Applying~\eqref{L_upper_bound} for $\mathcal{L}_L$, with the test function $\psi_b$ defined in~\eqref{test_b_1}, we have
\begin{align*}
   |J_3| &   = \frac{1}{\e^{s+k}} \int_s^t \int_{\O} \langle \mathcal{L}_L f ,\psi_b \rangle \dd x  \dd \tau   = \frac{1}{\e^{s+k}} \int_s^t \int_{\O}\langle \mathcal{L}_L (\mathbf{I}-\mathbf{P})f,\psi_b \rangle \dd x \dd \tau \\ 
   &\lesssim \frac{o(1)}{\e^s} \int_s^t {\color{black}\Vert b\Vert_{L^2_x}^2} + \frac{1}{\e^{s+2k}}\int_s^t \Vert (\mathbf{I}-\mathbf{P})f\Vert_\sigma^2 \dd \tau.
\end{align*}
Here we have used the fact that $\psi_b$ in~\eqref{test_b_1} contains a factor as $\mu^{1/2}$.

For the $a$ and $c$ we can deduce a similar conclusion. Then we conclude Corollary \ref{coro:landau}.

\end{proof}

\hide

\subsection{Cylindrical domain}\label{sec:cylinder}

In this section we consider a cylindrical domain, where $\O$ is periodic in $z$-coordinate. In Section \ref{sec_l6} to Section \ref{sec:proof}, two types of elliptic equation are applied to the weak formulation \eqref{weak_formula}. 

The first one is the standard Poisson equation with $0$ Neurmann boundary condition applied to the mass $a$ and energy $c$:
\begin{align*}
    \Delta \phi = a \text{ in } \O, \ \p_n \phi = 0 \text{ on }\p\O, \ \int_\O a \dd x = 0.
\end{align*}
From the standard Poincaré inequality, when $\int_\O \phi \dd x = 0$, we still have
\begin{align*}
  \Vert \phi\Vert_{L^p_x}  & \lesssim \Vert \nabla \phi\Vert_{L^p_x}.
\end{align*}
By standard Lax-Milgram theorem and elliptic estimate, we still have a unique solution $\phi$ such that
\begin{align*}
  \Vert \phi\Vert_{W^{2,p}_x}   & \lesssim \Vert \phi\Vert_{L^p_x}. 
\end{align*}

The other elliptic system is the symmetric Poisson system~\eqref{system}. In a cylindrical domain, additional condition need to be imposed for the Korn's inequality in Theorem \ref{lemma:estimate_nabla}. The condition is that we consider a Hilbert space as
\begin{align*}
    \chi_0 := \{u\in \O\times \mathbb{R}^3: u\in H^1_x, u_3=0, \ u\cdot n(x)=0 \text{ on }\p\O, \ P_\O(\nablaA u)=0\}.
\end{align*}
Then the Korn's inequality holds when assuming $u_3=0$ additionally, this gives the coercive estimate and the conditions of the Lax-Milgram theorem are satisfied. Thus Theorem \ref{thm:symmetric_poisson} holds with the modified Hilbert space.

Then for the estimate of the momentum $b$ in \eqref{system_b}, \eqref{system_b_5} and \eqref{system_b_5_sym}, under the modified Hilbert space, we have a unique solution $\phi_b$ with $\phi_b\cdot (0,0,1) = 0$. This extra condition does not have impact on the rest estimates, then we have the same conclusion in Theorem \ref{thm:coercive} and Theorem \ref{thm:l6}.

\unhide

\noindent \textbf{Acknowledgement.} CK thanks Professor Yan Guo for pointing out in February 2023 that the problem of this paper hasn't been solved by that time. The authors write this paper in response to his question affirmatively. CK is partly supported by NSF-DMS 1900923 and 2047681, and Simons fellowship. HC is partly supported by NSFC 8201200676 and GRF grant (2130715) from RGC of Hong Kong. HC thanks Professor Jiajun Tong for helpful discussion, and thanks Professor Zhennan Zhou for the kind hospitality
during his stay at Peking University. 

\hide

\unhide

\bibliographystyle{siam}
\bibliography{citation}

\end{document}